\documentclass[reqno]{amsart}

\usepackage{mathrsfs}
\usepackage{latexsym}
\usepackage{amsfonts}

\usepackage{hyperref}

\newtheoremstyle{myplain}{}{}{\it}
{0pt}{\scshape}{}{ }{\thmname{#1}\thmnumber{ #2}\thmnote{ (#3)}}
\newtheoremstyle{mydefinition}{}{}{}
{0pt}{\scshape}{}{ }{\thmname{#1}\thmnumber{ #2}\thmnote{ (#3)}}

\theoremstyle{myplain}

    \newtheorem{Def}{Definition}[section]
        \newtheorem{Lem}[Def]{Lemma}
        \newtheorem{theo}[Def]{Theorem}
        
        \newtheorem{rem}[Def]{Remark}
        \newtheorem{hyp}[Def]{Hypothesis}

\newcommand{\skp}[2]{\mbox{$\left\langle #1\, , \, #2\right\rangle$}}

\newcommand{\skpR}[2]{\mbox{$\left\langle #1\, ,\,#2\right\rangle_{L^2}$}}
\newcommand{\skpSi}[2]{\mbox{$\left\langle #1\, ,\,#2\right\rangle_{L^2(\Sigma)}$}}

\DeclareMathOperator{\dist}{dist}

\DeclareMathOperator{\supp}{supp}

\DeclareMathOperator{\Op}{Op}

\newcommand{\id}{\mathbf{1}}

\numberwithin{equation}{section}

\newcommand{\beqa}{\begin{eqnarray*}}
\newcommand{\eeqa}{\end{eqnarray*}}
\renewcommand{\hat}{\widehat}
\newcommand{\bauf}{\begin{itemize}}
\newcommand{\eauf}{\end{itemize}}
\newcommand{\be}{\begin{equation}}
\newcommand{\ee}{\end{equation}}
\newcommand{\ben}{\begin{enumerate}}
\newcommand{\een}{\end{enumerate}}
\newcommand{\ra}{\rightarrow}

\renewcommand{\O}{\Omega}
\newcommand{\ep}{\varepsilon}

\newcommand{\R}{{\mathbb R} }
\newcommand{\Z}{{\mathbb Z}}
\newcommand{\C}{{\mathbb C}}
\newcommand{\N}{{\mathbb N}}

\newcommand{\disk}{(\varepsilon {\mathbb Z})^d}

\newcommand{\Ce}{\mathscr C}

\newcommand{\De}{\mathscr D}

\newcommand{\E}{{\mathcal E}}

\newcommand{\B}{{\mathcal B}}
\newcommand{\TE}{\mathcal{T}}
\newcommand{\VE}{\mathcal{V}}
\newcommand{\te}{\mathfrak t}

\title{Agmon-type estimates for a class of jump processes}

\author{Markus Klein, Christian L\'eonard \and Elke Rosenberger}

\address{Markus Klein\\ Universit\"at Potsdam\\ Institut f\"ur Mathematik \\ Am Neuen Palais 10\\ 14469 Potsdam }
\email{mklein@math.uni-potsdam.de}
\address{Christian L\'eonard\\ Universit\'e Paris Ouest\\  }
\email{christian.leonard@u-paris10.fr}
\address{
Elke Rosenberger\\ Universit\"at Potsdam\\ Institut f\"ur Mathematik \\ Am Neuen Palais 10\\ 14469 Potsdam }
\email{erosen@uni-potsdam.de}

\date{\today}
\thanks{The authors gratefully acknowledge financial support of the Deutsch-Franz\"osische Hochschule through DFDK-01-06, which made this work possible.}



\begin{document}

\begin{abstract}
In the limit $\ep \to 0$ we analyze the generators $H_\ep$ of families of reversible jump processes in $\R^d$ associated with a class of 
symmetric non-local Dirichlet-forms and show exponential decay of the eigenfunctions. The exponential rate function is a Finsler distance, 
given as solution of a certain eikonal equation. Fine results are sensitive to the rate function being $\Ce^2$ or just Lipschitz. 
Our estimates are analog to the semiclassical Agmon estimates for differential operators of second order. 
They generalize and strengthen previous results on the lattice $\ep \Z^d$. Although our final interest is in the
(sub)stochastic jump process, technically this is a pure analysis paper, inspired by PDE techniques.
\end{abstract}

\maketitle

\section{Introduction}

We derive exponential decay results on eigenfunctions of a family of self adjoint generators $H_\ep,\, \ep\in (0,\ep_0],$  of (substochastic)
jump processes in $\R^d$ in the limit 
$\ep \to 0$. 
The basic idea behind our estimates is due to Agmon \cite{agmon}: The positivity of the quadratic form associated with a certain (weighted) operator
$H$ (on wavefunctions with support in a specific region) is related to exponential decay of solutions of $Hu=f$ in that region in weighted $L^2-$sense, 
and the (optimal) rate function admits a geometric interpretation as a geodesic distance (which is Riemannian if $H$ is a strongly elliptic differential 
operator of second order, but Finsler in our context). 

A semiclassical version of the Agmon estimate (for the Schr\"odinger operator) was developed in \cite{hesjo} by Helffer and Sj\"ostrand, who also 
applied such arguments in their analysis of Harper's equation \cite{harper} to a specific difference equation. In 
\cite{kleinro} a semiclassical Agmon estimate was proved for a classs of difference operators on the lattice $\ep \Z^d$, identifying the rate function as 
a Finsler distance.

We recall from \cite{hesjo} (for the Schr\"odinger operator) and from \cite{kleinro} - \cite{kleinro4} (for jump processes on the lattice $\disk$) 
that such estimates are an important first 
step to analyze the tunneling problem for a general multiwell problem, i.e. the appearance of exponentially close eigenvalues which can be
thought of as being produced by an interaction between the wells. It is our main goal to develop this analysis in the context of (substochastic) 
jump processes as 
considered in this paper. Our motivation comes from previous work on metastability (both for discrete Markov chains in \cite{begk1,begk2} and
continuous diffusions in \cite{begk3, begk4}). In both cases, exponential eigenvalue splitting for the low lying spectrum of the
(discrete or infinitesimal) generator has a direct interpretation in terms of exponentially long expected metastable transition times between
the wells. Conceptually, our small parameter $\ep$ plays the role of a semi-classical parameter $h$ (which
is usual in the context of Schr\"odinger´s equation). It seems that at present a PDE inspired approach to the tunneling problem (analog to 
the semiclassical  analysis for the Schr\"odinger operator) finally gives the sharpest results on the spectrum in cases of high regularity.
See \cite{hkn}, where the use of the semiclassical Witten complex sharpens the results in \cite{begk3, begk4} (originally proved via
potential theory) to full asymptotic expansions, in the case of continuous diffusions. For jump processes, even on $\disk$, final results at
this level of precision are not yet published (but see \cite{kleinro5} for the general substochastic case  on $\disk$, and the forthcoming
dissertation of G. Di Gesu \cite{Gia}, which develops the analog of the Witten complex for stochastic jump processes on $\disk$).

The jump processes considered in this paper are associated with non-local Dirichlet forms on the real Hilbert space $L^2(\R^d)$:
\begin{hyp}\label{Hypo1}
Let $\E_\ep,\, \ep\in (0,\ep_0],$ be a family of bilinear forms on $L^2(\R^d, dx)$ with domains $\De(\E_\ep)$ given by
\begin{align}\label{Eform}
\E_\ep (u,v) &:= \frac{1}{2} \int_{\R^d} \!dx \int_{\R^d\setminus\{0\}}\! (u(x) - u(x+\ep \gamma)) (v(x) - v(x+\ep \gamma))\, 
K_\ep (x, d\gamma) + \int_{\R^d} V_\ep (x) u(x) v(x) \, dx \\
 \De(\E_\ep) &= \{ u\in L^2 (\R^d, dx)\, |\,  \E_\ep (u,u) < \infty \} \nonumber\; ,
\end{align}
where for all $\ep\in(0, \ep_0]$
\ben
\item $V_\ep(x)\, dx$ is a positive Radon measure on $\R^d$ 
\item for $x\in\R^d$, $K_\ep(x,\, .\,)$ is a positive Radon measure on the Borel sets $\B (\R^d\setminus\{0\})$
satisfying
\ben
\item[(i)]
$K_\ep (x, E) < \infty$ for all $E\in \B (\R^d\setminus\{0\})$ with $\dist (E, 0) \geq \delta >0$ 
\item[(ii)] $\int_{|\gamma | \leq 1} |\gamma|^2 K_\ep(x, d\gamma) \leq C$ locally uniformly in $x\in\R^d$
\item[(iii)] $K_\ep (x, d\gamma)\, dx$ is a reversible measure on $Y:= \R^d\times \R^d\setminus\{0\}$ in the sense that for all
non-negative $\phi, \psi \in \Ce_0(\R^d)$
\begin{equation}\label{revers}
 \int_{Y} \phi (x+\ep \gamma) \psi (x) K_\ep (x,\, d\gamma)\, dx =  \int_{Y} \phi (x) \psi (x+\ep \gamma) K_\ep (x,\, d\gamma)\, dx\; .
\end{equation}
\een
\een
\end{hyp}

We shall formally denote the reversibility condition \eqref{revers} as 
\begin{equation}\label{revers2}
 K_\ep (x,\, d\gamma)\, dx = K_\ep (x+\ep \gamma,\, -d\gamma) \, dx\; ,
\end{equation}
where the right hand side denotes the Radon measure on $Y$ given by
\[
 \int_{Y} f(x, \gamma) K_\ep (x + \ep \gamma, -d\gamma) \, dx := \int_{Y} f(x, -\gamma) K_\ep (x + \ep \gamma, d\gamma) \, dx 
:= \int_{Y} f(x - \ep \gamma, -\gamma) K_\ep (x, d\gamma) \, dx \; ,
\]
and (abusing notation) we shall even cancel $dx$ on both sides of \eqref{revers2}.\\

Assuming Hypothesis \ref{Hypo1}, $\E_\ep$
is a Dirichlet form (i.e. closed, symmetric and Markovian) and $\Ce^\infty_0(\R^d) \subset \De(\E_\ep)$ for all $\ep\in (0,\ep_0]$ (see 
Fukushima-Oshima-Takeda \cite{Fu}).

The general theory of Dirichlet forms $\E$ analyzed in \cite{Fu} covers the case, where $(\R^d, dx)$ is replaced by $(X, m)$ if $X$ is a locally compact 
separable metric space and $m$ is a positive Radon measure on $X$ with $\supp  m = X$, provided that $\De(\E)$ is dense in
$L^2(X, m)$.

In particular, this is true for $X=(\ep \Z)^d$ and $m$ being the counting measure on $X$. In this situation, we proved similar decay results in
\cite{kleinro}
and \cite{thesis}, with $K_\ep(x, m(d\gamma)) = - a_{\ep\gamma}(x; \ep) m(d(\ep\gamma))$ as a measure on $\Z^d\setminus\{0\}$
(in fact, we treated a slightly more general case where the form $\E_\ep$ instead of being positive is only semibounded,
$\E_\ep (u,u) \geq -C\ep$ for some $C>0$ and $\ep\in (0,\ep_0]$). 

In this paper, we focus on the complementary and much more singular case $X=\R^d$. We remark that, combining the results
of \cite{kleinro} with this paper, one could treat the case where $X$ is an arbitrary closed subgroup of $(\R^d, +)$ and $m$ is Haar measure
on $X$. Since our methods depend on some elements of Fourier analysis, this is a natural framework for our results.

To control the limit $\ep\to 0$, we shall impose stronger conditions on $K_\ep$ and $V_\ep$.

\begin{hyp}\label{Hypo2}
\begin{enumerate}
\item The measure $K_\ep(x,\, . \, )$ satisfies
\begin{equation}\label{entwK}
 K_\ep(x, \, . \,) = K^{(0)}(x, \, . \,) + R^{(1)}_\ep (x, \, . \,) \,
\qquad (x \in\R^d)\, ,
\end{equation}
where
\begin{enumerate}
\item[(i)] for any $c>0$ there exists $C>0$ such that uniformly
with respect to $x\in\disk$ and $\ep\in (0,\ep_0]$
\begin{align}\label{abfallagamma1}
\int_{|\gamma | \geq 1} e^{c|\gamma|} K^{(0)}(x, d\gamma) \leq C \qquad &\text{and} \qquad
\int_{|\gamma |\geq 1} e^{c|\gamma|} \bigl|R_\ep^{(1)}(x, d\gamma)\bigr|  \leq C\ep \\
\int_{|\gamma | \leq 1} |\gamma|^2 K^{(0)}(x, d\gamma)  \leq C \qquad &\text{and} \qquad
\int_{|\gamma | \leq 1}  |\gamma|^2 \bigl|R_\ep^{(1)}(x, d\gamma)\bigr|  \leq C\ep \label{abfallagamma2}
\end{align}
\item[(ii)] for all $x\in\R^d$ there exists $c_x>0$ such that for all $v\in\R^d$
\begin{equation}\label{convexab} 
\int_{\R^d\setminus\{0\}} (\gamma \cdot v)^2 K^{(0)} (x, d\gamma) \geq c_x \|v\|^2 \; . 
\end{equation}
\end{enumerate}
\item 
\ben
\item[(i)]
The potential energy $V_\ep\in\Ce^2 (\R^d, \R)$ satisfies
\[
V_{\ep}(x) = V_0 (x) + R_{1}(x;\ep)  \, ,
\]
where $V_0\in\Ce^2(\R^d)$, $R_{1}\in \Ce^2 (\R^d\times (0,\ep_0])$ 
and for any compact set $K\subset \R^d$
there exists a constant $C_K$ such that $\sup_{x\in K} |R_{1}(x;\ep)|\leq C_K \ep$.
\item[(ii)]
$V_0(x)\geq 0$ and it takes the value $0$ only at a finite number of non-degenerate minima
$x_j,$  i.e. 
$D^2 V_0|_{x_j} >0, \; j\in \mathcal{C} =\{1,\ldots , r\}$,
which we call potential wells.
\een
\end{enumerate}
\end{hyp}

We remark that combining the positivity of the measure $K_\ep(x, \, .\,)$ with Hypothesis \ref{Hypo2}(a), it follows that
$K^{(0)}(x,\, . \,)$ is positive while $R^{(1)}_\ep(x,\, . \,)$ is possibly signed.

It is well known (see e.g. \cite{Fu}) that $\E_\ep$ uniquely determines a self adjoint operator $H_\ep$ in $L^2(\R^d)$. 
To introduce Dirichlet boundary conditions for $H_\ep$ on some open set $\Sigma\subset \R^d$, one considers the form
\begin{equation}\label{Dirichlet1} 
\tilde{\E}_\ep^\Sigma (u,v) = \E_\ep (u,v)\quad\text{with domain}\quad \De ( \tilde{\E}_\ep^\Sigma)=\Ce_0^\infty(\Sigma)\, .
\end{equation} 
Then $\tilde{\E}_\ep^\Sigma$ is Markovian (see \cite{Fu}, ex. 1.2.1) and closable. In fact, if we consider $L^2(\Sigma)$ as a subset of $L^2(\R^d)$ 
(extend
$f\in L^2(\Sigma)$ to $\R^d$ by zero), the form
\begin{equation}\label{Neumann1} 
\hat{\E}_\ep^\Sigma (u,v) = \E_\ep(u,v)\quad\text{with domain}\quad \De(\hat{\E}_\ep^\Sigma) = \{u\in L^2(\Sigma)\, |\, \ \hat{\E}_\ep^\Sigma (u,u) 
< \infty\}
\end{equation}
- corresponding to Neumann boundary conditions - is a closed (Markovian) extension of $\tilde{\E}_\ep^\Sigma$ (see \cite{Fu}, ex. 1.2.4), giving
closability of $\tilde{\E}_\ep^\Sigma$.

\begin{Def}\label{Def1} We denote by $\E^\Sigma_\ep$ the closure of $\tilde{\E}^\Sigma_\ep$ given in \eqref{Dirichlet1}. The 
operator $H_\ep^\Sigma$ with Dirichlet boundary conditions on $\Sigma$ is the unique self adjoint operator associated to $\E^\Sigma_\ep$.
The unique self adjoint operator $\hat{H}_\ep^{\Sigma}$ associated to $\hat{\E}_\ep^\Sigma$ defined in \eqref{Neumann1} 
represents Neumann boundary conditions on $\Sigma$.
\end{Def}

By \cite{Fu}, Thm. 3.1.1, $\E^\Sigma_\ep$ is Markovian (as the closure of a Markovian form) and thus a Dirichlet form.
In particular, since $\E^\Sigma_\ep$ is a restriction of $\hat{\E}^\Sigma_\ep$, we have for $u,v\in\De (\E_\ep^\Sigma)$,
\begin{align}\label{ESigma}
 \E^\Sigma_\ep (u,v) &= \TE^\Sigma_\ep (u,v) + \VE^\Sigma_\ep (u,v) \qquad\text{with}\\
\TE_\ep^\Sigma (u,v) &:= \frac{1}{2} \int_{\Sigma} dx \int_{\Sigma'(x)} (u(x) - u(x+\ep \gamma)) (v(x) - v(x+\ep \gamma))\, 
K_\ep (x, d\gamma) \quad\text{and} \label{TSigma} \\
\VE_\ep^\Sigma (u,v) &:= \int_{\Sigma} V_\ep (x) u(x) v(x) \, dx \label{VSigma}\\
\text{where}&\quad \Sigma'(x) := \{\gamma\in\R^d\setminus\{0\}\,|\, x+\ep\gamma\in\Sigma\}\label{Sigma'}.
\end{align}
Similarly, $\hat{\E}_\ep^\Sigma = \hat{\TE}_\ep^\Sigma + \hat{\VE}_\ep^\Sigma$.
We remark that $\TE_\ep^\Sigma$, $\VE_\ep^\Sigma$. $\hat{\TE}_\ep^\Sigma$ and $\hat{\VE}_\ep^\Sigma$ are again Dirichlet forms,
in particular they are positive.\\
We will use the notation $q[u] := q(u,u)$ for the quadratic form associated to any bilinear form $q$.\\

Concerning the operator $H_\ep$ associated to $\E_\ep$ we remark that, even assuming Hypothesis \ref{Hypo2} in addition to Hypothesis 
\ref{Hypo1}, it is far from trivial to characterize the domains
$\De(H_\ep)$ and $\De(H_\ep^\Sigma)$. Without additional assumptions, $H_\ep =: T_\ep + V_\ep$ (or $H_\ep^\Sigma$) is not even defined on 
$\Ce_0^\infty(\R^d)$ (or $\Ce_0^\infty(\Sigma)$ resp.). 
However, there are some cases for which we can give formulae for $T_\ep u$ on subsets of its domain.
\ben
\item If the measure $K_\ep(x,\, \,.\,)$ is finite uniformly with respect to $x\in\R^d$, one has
\[
 T_\ep u(x) = \int_{\R^d\setminus\{0\}} (u(x) - u(x+\ep \gamma)) K_\ep (x,\, d\gamma) 
\]
and $T_\ep$ is bounded on $L^2(\R^d)$.
\item 
If $K_\ep(x, d\gamma) = k_\ep (x, \gamma)\, d\gamma$, where $k_\ep \in \Ce (\R^d\times \R^d\setminus\{0\})$ is
Lipschitz in $x\in\R^d$, locally uniformly with respect to $\gamma\in \R^d\setminus\{0\}$, then 
$\Ce^2_0(\R^d)\subset \De(H_\ep)$ and, for $u\in \Ce^2_0(\R^d)$, 
\begin{multline}\label{operator2}
 T_\ep u(x) = \int_{\R^d\setminus\{0\}} \bigl(2 u(x) - u(x+\ep\gamma) - u(x-\ep\gamma)\bigr) k_\ep(x, \gamma) \, d\gamma \\
+\int_{\R^d\setminus\{0\}} \bigl(u(x) - u(x - \ep\gamma) \bigr) \bigl(k_\ep(x-\ep\gamma, \gamma) - k_\ep (x, \gamma)\bigr) \, d\gamma\; .
\end{multline}
\item 
If $K_\ep(x, d\gamma) = K_\ep (x, -d\gamma)$, then $\Ce_0^2(\R^d)\subset \De (H_\ep)$ and for $u\in \Ce_0^2(\R^d)$ one even 
has  the simpler form
\begin{equation}\label{operator3}
 T_\ep u(x) = \int_{\R^d\setminus\{0\}} \bigl( u(x) - u(x + \ep\gamma) - \ep \gamma \nabla u(x)\bigr) K_\ep (x, d\gamma)\; .
\end{equation}
\item In the case of a Levy-process, i.e. if $K_\ep (x, \, d\gamma) = K_\ep (d\gamma)$ (which by reversibility, see \eqref{revers2}, implies 
$K_\ep (d\gamma) = K_\ep (-d\gamma)$) one has both the representation \eqref{operator3} and (since the second term
on the rhs of \eqref{operator2} formally vanishes)
\[
 T_\ep u(x) = \int_{\R^d\setminus\{0\}} \bigl(2 u(x) - u(x+\ep\gamma) - u(x-\ep\gamma)\bigr) K_\ep(d\gamma) \; .
\]
\een
Similar formulae hold for the operators with Dirichlet (and Neumann) boundary conditions. In this paper, we shall need none of them,
since we shall  directly work  with the Dirichlet form \eqref{Eform}.\\

We define $t_0: \R^{2d} \rightarrow \R$ as
\begin{equation}\label{t_0}
t_0(x,\xi) := \int_{\R^d\setminus\{0\}} \Bigl(1- \cos \bigl(\eta\cdot \xi\bigr)\Bigr) K^{(0)}(x, d\gamma)  \; ,
\end{equation}
which in view of Hypothesis \ref{Hypo2}(a),(ii) extends to an entire function in $\xi\in\C^d$, and we set
\begin{equation}\label{tildetdef}
\tilde{t}_0(x, \xi) := - t_0(x, i\xi) = \int_{\R^d\setminus\{0\}} \Bigl(\cosh \bigl(\eta\cdot \xi\bigr) - 1\Bigr) K^{(0)}(x, d\gamma)  \, , \qquad (x,\xi\in\R^d)\; .
\end{equation}

We remark that $t_0$ formally is the principal symbol $\sigma_p (T_\ep)$ - the leading order term in $\ep$ of the symbol - associated to the 
operator $T_\ep$ 
under semiclassical quantization (with $\ep$ as small parameter). Recall that for a symbol 
$b\in \Ce^\infty(\R^{2d}\times (0,\ep_0))$, the corresponding operator
is (formally) given by
\[
\Op_\ep(b)\, v(x) := (\ep 2\pi)^{-d} \int_{\R^{2d}} e^{\frac{i}{\ep}(y-x)\xi}\,
b(x,\xi;\ep)v(y) \, dy \, d\xi \, ,\qquad v\in\mathscr{S}(\R^d)\, ,
\]
(for details on pseudo-differential operators see e.g. Dimassi-Sj\"ostrand \cite{dima}).

In particular, the translation operator $\tau_{\pm\ep\gamma}$ acting as $\tau_{\pm\ep\gamma} u(x) = u(x \pm\ep\gamma)$, has
the $\ep$-symbol $e^{\mp i\gamma\xi}$. Thus, writing $T_\ep$ formally as
\[ 
\frac{1}{2}\int_{\R^d\setminus\{0\}} \bigl(\id - \tau_{-\ep\gamma}\bigr) K_\ep (x, d\gamma) \bigl(\id - \tau_{\ep\gamma}\bigr) 
\]
and using $\sigma_p(A \circ B) = \sigma_p(A) \sigma_p(B)$ for the principal symbols of operators $A, B$, immediately
gives $t_0 = \sigma_p(T_\ep)$ given in \eqref{t_0}.

We emphasize, however, that under the weak regularity assumptions given in Hypotheses \ref{Hypo1} and \ref{Hypo2}, $T_\ep$ is
not an honest pseudo-differential operator (i.e. one with a $\Ce^\infty$-symbol, for which the symbolic calculus holds), but only a
quantization of a singular symbol, giving a map ${\mathscr S}(\R^d) \rightarrow {\mathscr S}' (\R^d)$ (see \cite{dima}). 

We shall now assume

\begin{hyp}\label{Hypo3}
Given Hypotheses \ref{Hypo1} and \ref{Hypo2}, $\Sigma\subset \R^d$ is an open bounded set with $x_j\in \Sigma$ for exactly one 
$j\in\mathcal{C}$ and 
$x_k\notin \overline{\Sigma}$ for $k\in\mathcal{C}, k\neq j$. Moreover there is an open set $\O\subset\Sigma$ containing $x_j$ and a 
Lipschitz-function
$d: \overline{\Sigma}\rightarrow [0,\infty))$ satisfying, for $\tilde{t}_0$ defined in \eqref{t_0},
\ben
\item $d(x_j)=0$ and $d(x)\neq 0$ for $x\neq x_j$.
\item $d\in \Ce^2(\overline{\O})$.
\item the (generalized) eikonal equation holds  in some neighborhood $U\subset\O$ of $x_j$, i.e.
\begin{equation}\label{eikonal}
 \tilde{t}_0(x, \nabla d(x)) = V_0(x)\qquad\text{for all}\quad x\in U\; .
\end{equation}
\item the (generalized) eikonal inequality holds in $\Sigma$, i.e.
\begin{equation}\label{eikonalun}
\tilde{t}_0(x, \nabla d(x)) - V_0(x) \leq 0 \qquad\text{for all}\quad x\in \Sigma\; .
\end{equation}
\een
 \end{hyp}

We remark that in a more regular setting, i.e. if $\tilde{h}_0:= \tilde{t}_0 - V_0\in\Ce^\infty (\R^{2d})$, such a function $d$ may be
constructed  as a distance in a certain Finsler metric associated with $\tilde{h}_0$ (see \cite{kleinro}, which introduces Finsler
metrics for hyperregular Hamiltonians on $T^*M$, where $M$ is a $\Ce^\infty$-manifold), if $\Sigma$ avoids the
cut locus. We recall that for some special Hamilton functions on $\R^{2d}$ Finsler distances were used in \cite{matte} to solve
Hamilton-Jacobi equations (but not to describe the decay of eigenfunctions), similarly to the somewhat older results in \cite{tinta}.
In both papers, the cut-locus (where the Finsler distance is only Lipschitz) is excluded in the applications by assumption.
The solution of Hamilton-Jacobi equations in $\R^{2d}$ by Finsler distances in very general situations of low regularity is discussed
e.g. in \cite{sico}. 
We shall discuss the Finsler distance $d$ in the case of low regularity and its relation to large deviation results for jump
processes (see e.g. \cite{leonard}) in a future publication, including the relation between the Hamiltonian and Lagrangian point of view.

Our central results are the following theorems on the decay of eigenfunctions of $H_\ep^\Sigma$ and $\hat{H}_\ep^\Sigma$.

\begin{theo}\label{weig}
Assume Hypotheses \ref{Hypo1}, \ref{Hypo2} and \ref{Hypo3} with $\Sigma = \O$.
Let $H_\ep^\Sigma$ and $\hat{H}_\ep^{\Sigma}$ be the operators with Dirichlet and Neumann boundary conditions from Definition \ref{Def1}.  

Fix $R_0>0$.
Then there exist constants $\ep_0, B, C>0$ such that for all $\ep\in(0,\ep_0]$, $E\in [0,\ep R_0]$
and real $u\in \De(H_\ep^\Sigma)$
\begin{equation}\label{weigequ}
\left\| \left(1+\tfrac{d}{\ep}\right)^{-B} e^{\frac{d}{\ep}} u
\right\|_{L^2(\Sigma)}
\leq
C \Bigl[ \ep^{-1}\left\| \left(1+\tfrac{d}{\ep}\right)^{-B}
e^{\frac{d}{\ep}}
\left(H_\ep^{\Sigma}-E\right)u\right\|_{L^2(\Sigma)} +
 \| u \|_{L^2(\Sigma)}  \Bigr]
\; .
\end{equation}
In particular, let
$u\in\De(H_\ep^\Sigma) $ be an eigenfunction of $H_\ep^\Sigma$
with respect to the eigenvalue
$E\in[0,\ep R_0]$,
then 
\begin{equation}\label{eigenu}
 \left\| \left(1+\tfrac{d}{\ep}\right)^{-B} e^{\frac{d}{\ep}}u\right\|_{L^2(\Sigma)} \leq
 C \| u \|_{L^2(\Sigma)}\, .
 \end{equation}
Analog results hold for $u\in\De (\hat{H}_\ep^\Sigma)$ and $u$ a normalized eigenfunction of  $\hat{H}_\ep^\Sigma$ respectively.\\
\end{theo}

The following theorem gives a weaker result in the case that $d$ is only Lipschitz outside some small ball around $x_j$. Thus the cut-locus
is allowed to meet $\Sigma$.
Then we have to assume more regularity of $K^{(0)}$ with respect to $x$. 

\begin{theo}\label{Theorem2}
 Assume Hypotheses \ref{Hypo1}, \ref{Hypo2} and \ref{Hypo3} and let $H_\ep^\Sigma$ and $\hat{H}_\ep^{\Sigma}$ be the 
operators with Dirichlet and Neumann boundary conditions from Definition \ref{Def1}. Moreover assume that 
$K^{(0)}(\,.\,, d\gamma)$ is continuous in the sense that for all $c>0$
\begin{equation}\label{K0stetig}
 \int_{\gamma \in\R^d\setminus \{0\}}|\gamma|^2 e^{c|\gamma|} \bigl( K^{(0)}(x+h, d\gamma) - K^{(0)}(x, d\gamma)\bigr) =
 o(1)\qquad (|h|\to 0)
\end{equation}
locally  uniformly in $x\in\R^d$. Fix $R_0>0$ and a constant $D>0$ such that the ball $K:= \{x\in\R^d\,|\, d(x)\leq D\}$ is contained in $\O$. 
Then 
\ben
\item for any fixed $\alpha\in (0,1]$ there exist constants $C$, $\ep_\alpha>0$ such that for all $\ep\in (0,\ep_\alpha]$, $E\in [0,\ep R_0]$
and real $u\in \De (H_\ep^\Sigma)$
\begin{equation}\label{ab1}
\left\|  e^{\frac{(1-\alpha)d}{\ep}} u \right\|_{L^2(\Sigma)}
\leq C \Bigl[ \ep^{-1}\left\| e^{\frac{d}{\ep}}
\left(H_\ep^{\Sigma}-E\right)u\right\|_{L^2(\Sigma)} +
 \| u \|_{L^2(\Sigma)}  \Bigr] \; .
\end{equation}
\item there exists constants $C', B>0$ such that for any $\alpha\in (0,1]$ there exists $\Phi_\alpha\in\Ce^2(\overline{\Sigma})$ such that
\begin{align}\label{abinK}
 e^{\frac{d(x)}{\ep}}\frac{1}{C'}\left(1+\tfrac{d(x)}{\ep}\right)^{-\frac{B}{2}} &\leq
e^{\frac{\Phi_\alpha(x)}{\ep}} \leq
e^{\frac{d(x)}{\ep}} C'\left(1+\tfrac{d(x)}{\ep}\right)^{-\frac{B}{2}}\quad\text{for}\;\, x\in K\;\,\text{and}\\
e^{\frac{(1-\alpha) d(x)}{\ep}} & \leq e^{\frac{\Phi_\alpha(x)}{\ep}} \leq
e^{\frac{d(x)}{\ep}}\qquad\qquad\text{for}\;\, x\in \overline{\Sigma}\setminus K\label{abohneK}\; ,
\end{align}
and such that there exists $\alpha_0\in (0,1]$ such that for any 
$\alpha\in (0,\alpha_0]$ there exists constants $C$, $\ep_\alpha>0$ such that for all $\ep\in(0,\ep_\alpha]$, 
$E\in [0,\ep R_0]$ and real $u\in \De(H_\ep^\Sigma)$
\begin{equation}\label{weigequ2}
\left\|  e^{\frac{\Phi_\alpha}{\ep}} u \right\|_{L^2(\Sigma)}
\leq C \Bigl[ \ep^{-1}\left\| e^{\frac{\Phi_\alpha}{\ep}}
\left(H_\ep^{\Sigma}-E\right)u\right\|_{L^2(\Sigma)} +
 \| u \|_{L^2(\Sigma)}  \Bigr] \; .
\end{equation}
Moreover for any $\alpha\in (0, 1]$ there exists $\ep_\alpha>0$ such that for any $\ep\in (0,\ep_\alpha]$, $E\in [0,\ep R_0]$ and real 
$u\in\De(H_\ep^\Sigma)$
\begin{equation}\label{thm161}
\frac{1}{C'}  
\left\| \left(1+\tfrac{d}{\ep}\right)^{-\frac{B}{2}} e^{\frac{d}{\ep}}u\right\|^2_{L^2(K)} + 
 \left\| e^{\frac{(1-\alpha) d(x)}{\ep}}u\right\|^2_{L^2(\Sigma\setminus K)}\leq
\left\|  e^{\frac{\Phi_\alpha}{\ep}} u \right\|^2_{L^2(\Sigma)}
\end{equation}
and if $u$ is an eigenfunction of $H_\ep^\Sigma$ with respect to the eigenvalue $E\in[0,\ep R_0]$, then 
\begin{equation}\label{eigenu2}
\left\|  e^{\frac{\Phi_\alpha}{\ep}} u \right\|_{L^2(\Sigma)} \leq C \| u \|_{L^2(\Sigma)} \, ,
 \end{equation}
for $C$ given in \eqref{weigequ2}.
\een 
Analog results hold for $\hat{H}_\ep^\Sigma$ and for real $u\in\De (\hat{H}_\ep^\Sigma)$ respectively.\\
\end{theo}

\begin{rem}\label{rem1}
 All assertions of Theorem \ref{weig}  and \ref{Theorem2} remain true if $\E_\ep$ is not necessarily positive, but only satisfies $\E_\ep (x) 
 \geq - C\ep$ or, more special, $\E_\ep\geq 0$ but $V_\ep \geq -C\ep$.
In a stochastic context, such a situation could arise if e.g. one starts with a Dirichlet form $\tilde{\E}_\ep$ on
$L^2(m_\ep)$ associated with a pure jump process 
(with $V_\ep =0$), given by a kernel
$\tilde{K}_\ep(x, d\gamma)$, which is integrable with respect to $\gamma\in\R^d\setminus\{0\}$,  i.e. satisfies 
$\int \tilde{K}_\ep (x,d\gamma) <\infty$, and reversible with respect to $m_\ep(dx) = e^{-\frac{F(x)}{\ep}} dx$. If 
$K_\ep (x, d\gamma) := e^{\frac{F(x + \ep\gamma) - F(x)}{2\ep}}
\tilde{K}_\ep(x, d\gamma)$ is integrable with respect to $\gamma\in\R^d\setminus\{0\}$, then
\[ 
\E_\ep (u,v) := \tilde{\E}_\ep \bigl(e^{\frac{F}{2\ep}} u, e^{\frac{F}{2\ep}} v\bigr) = 
\int_{\R^d}\int_{\R^d\setminus\{0\}} (u(x+\ep\gamma) - u(x))(v(x+\ep\gamma) - v(x))\, K_\ep (x, d\gamma)\, dx + 
\skpR{u}{V_\ep v} \, ,
\]
is a Dirichlet form on $L^2(dx)$, where 
\[ 
V_\ep (x) =  \int_{\R^d\setminus\{0\}} \bigl(e^{-\frac{F(x + \ep\gamma) - F(x)}{2\ep}} - 1\bigr) K_\ep (x, d\gamma) = 
\int_{\R^d\setminus\{0\}} (\tilde{K}_\ep - K_\ep) (x, d\gamma) \; .
\]
If $F$ is smooth and $K_\ep$ and $\tilde{K}_\ep$ have an expansion as in Hypothesis \ref{Hypo2}(a), then one verifies that
$K^{(0)} (x, d\gamma) = K^{(0)}(x, -d\gamma)$ and $V_\ep \geq -C\ep$ for some constant $C>0$. If the integrability conditions for
$K_\ep$ and $\tilde{K}_\ep$ are not satisfied, the above transformation is more delicate and requires regularity of 
$K_\ep (x, d\gamma)$ in $x$.
\end{rem}
 We emphasize that the eigenvalue $E$ in Theorem \ref{weig}  and \ref{Theorem2} need not be discrete (a priori, it could be of infinite
 multiplicity or be imbedded into the continuous (or essential) spectrum of $H_\ep$). In this paper, $H_\ep$ need not have a
 spectral gap. 
 However, to develop tunneling theory in analogy to \cite{kleinro, kleinro2, kleinro3, kleinro4}, one needs to impose further conditions on
 the jump kernel $K_\ep$.

\section{Preliminary Results}

This section contains preparations for the proof of Theorem \ref{weig}  and \ref{Theorem2}. Lemmata \ref{Hepconj}
- \ref{HepDchi} contain our abstract approach to Agmon type estimates, while Lemmata \ref{propt} - \ref{tnullVnull} contain more specific 
estimates on $\tilde{t}_0(x, \xi), d(x)$ 
and the phasefunctions used in the proof of Theorem \ref{weig}  and \ref{Theorem2}.

\begin{Lem}\label{Hepconj}
Assume Hypotheses \ref{Hypo1} and \ref{Hypo2}  and,
for $\Sigma\subset\R^d$ open, let $\E_\ep^\Sigma$ and $\hat{\E}_\ep^\Sigma$ denote the associated Dirichlet forms given in Definition \ref{Def1} 
and \eqref{Neumann1} respectively.
Let $\varphi:\R^d \ra \R$ be Lipschitz and bounded.
Then for any real valued $v$ with $e^{\pm\frac{\varphi}{\ep}} v \in \De (\E^\Sigma_\ep)$ (or $ \De (\hat{\E}^\Sigma_\ep)$ resp.)
\begin{multline}\label{lemma2}
 \E_\ep^\Sigma \bigl(e^{-\frac{\varphi}{\ep}} v, \,e^{\frac{\varphi}{\ep}}v\bigr)=
\skpSi{\bigl(V_\ep + V_{\ep,\Sigma}^\varphi \bigr) v}{v}  \\
 + \frac{1}{2}\int_{\Sigma}\,dx \int_{\Sigma'(x)}\, K_\ep(x, d\gamma)
\cosh \Bigl(\tfrac{1}{\ep}\bigl(\varphi (x)-\varphi (x+\ep\gamma)\bigr)\Bigr)
\bigl(v(x)-v(x+\ep\gamma)\bigr)^2 \, ,
\end{multline}
where $\Sigma'(x)$ is defined in \eqref{Sigma'} and
\begin{equation}\label{Vphiep}
V_{\ep,\Sigma}^\varphi (x) :=
\int_{\Sigma'(x)}\left[ 1 -  \cosh \Bigl(\tfrac{1}{\ep}
\bigl(\varphi(x)- \varphi(x+\ep\gamma)\bigr)\Bigr)\right] K_\ep(x, d\gamma)\; ,
\end{equation}
which ist bounded uniformly in $\ep$. An analog result holds for $\hat{\E}_\ep^\Sigma$.
\end{Lem}

\begin{proof}
We have by \eqref{ESigma}
\begin{multline}\label{lemma2-1}
 \E_\ep^\Sigma \bigl(e^{-\frac{\varphi}{\ep}} v, \,e^{\frac{\varphi}{\ep}}v\bigr) - \skpSi{V_\ep v}{v}  \\
= \frac{1}{2}\int_\Sigma dx \int_{\Sigma'(x)} \Bigl( v(x)^2 - 2 \cosh \bigl(\tfrac{1}{\ep}(\varphi (x+\ep\gamma) - \varphi (x))\bigr) v(x) v(x+\ep\gamma) + v(x+\ep\gamma)^2\Bigr) K_\ep(x, \,d\gamma) \\
 = \frac{1}{2} \int_\Sigma dx \int_{\Sigma'(x)} \bigl( v(x)^2 + v(x + \ep\gamma)^2\bigr) 
\Bigl(1 -  \cosh \bigl(\tfrac{1}{\ep}(\varphi (x+\ep\gamma) - \varphi (x))\bigr) \Bigr) K_\ep(x, \,d\gamma) \\
 +  \frac{1}{2}\int_\Sigma dx \int_{\Sigma'(x)}  \cosh \bigl(\tfrac{1}{\ep}(\varphi (x + \ep\gamma) - \varphi (x))\bigr) \bigl( v(x) - 
 v(x+\ep\gamma)\bigr)^2 K_\ep(x, \,d\gamma) \; .
\end{multline}
Since $\cosh \xi$ is even with respect to $\xi$ and by the reversibility \eqref{revers} of $K_\ep (x, d\gamma)$
\begin{multline}\label{lemma2-3}
  \frac{1}{2} \int_\Sigma dx \int_{\Sigma'(x)} \bigl( v(x)^2 + v(x + \ep\gamma)^2\bigr) 
\Bigl(1 -  \cosh \bigl(\tfrac{1}{\ep}(\varphi (x+\ep\gamma) - \varphi (x))\bigr) \Bigr) K_\ep(x, \,d\gamma)  \\ = 
\int_\Sigma dx \int_{\Sigma'(x)}  v(x)^2 
\Bigl(1 -  \cosh \bigl(\tfrac{1}{\ep}(\varphi (x+\ep\gamma) - \varphi (x))\bigr) \Bigr) K_\ep(x, \,d\gamma) \; .
\end{multline}
Thus inserting \eqref{lemma2-3} into \eqref{lemma2-1} and using the definition of $V_{\ep, \Sigma}^\varphi$ gives \eqref{lemma2}.

To show boundedness of the integral on the right hand side of \eqref{Vphiep}, one observes that $\cosh t - 1 \leq |t| \sinh |t|$ for
all $t\in\R$. Choosing $t= \frac{1}{\ep}(\varphi (x) - \varphi (x+\ep\gamma))$ and using that $\varphi$ is Lipschitz with Lipschitz constant $L>0$ gives 
\begin{equation}\label{lemma2-4}
 \cosh \Bigl(\tfrac{1}{\ep}
\bigl(\varphi(x)- \varphi(x+\ep\gamma)\bigr)\Bigr) - 1 \leq  L^2 |\gamma|^2 \frac{\sinh (L|\gamma|)}{L|\gamma|}\; .
\end{equation}
Inserting \eqref{lemma2-4} into \eqref{Vphiep} proves the assertion, according to 
Hypothesis \ref{Hypo2},(a),(i).

Since the formula \eqref{ESigma} also holds for $\hat{\E}_\ep^\Sigma$, the same arguments give the analog result.\\ 
\end{proof}

\begin{Lem}\label{lemma3}
 Assume Hypotheses \ref{Hypo1} and \ref{Hypo2} and for $\Sigma\subset\R^d$ open, let $\E_\ep^\Sigma$, $\hat{\E}_\ep^\Sigma$ and
$\varphi$ be as in Lemma \ref{Hepconj}. Then 
\[  v \in \De (\E^\Sigma_\ep)\; \text{or}\;  \De (\hat{\E}^\Sigma_\ep)\;\text{resp.}  \quad\Rightarrow \quad
e^{\frac{\varphi}{\ep}} v \in \De (\E^\Sigma_\ep)\;\text{or}\;  \De (\hat{\E}^\Sigma_\ep)\;\text{resp.} \; . \]
\end{Lem}

\begin{proof}
We will use the notation (see \eqref{ESigma})
\begin{equation}\label{lemma3-5}
 \te^\Sigma_\ep[u] := \hat{\E}^\Sigma_\ep [u] + \|u\|^2_{L^2(\Sigma)} = \hat{\TE}_\ep^\Sigma[u] + \hat{\VE}_\ep^\Sigma[u] +
\|u\|^2_{L^2(\Sigma)}\; .  
\end{equation}
We recall that a function $f\in \De (\hat{\E}^\Sigma_\ep)$ is in $\De (\E^\Sigma_\ep)$, if and only if there is a sequence $(f_n)_{n\in\N}$
in $\De (\tilde{\E}^\Sigma_\ep)$ such that $ \te^\Sigma_\ep [f_n - f]\rightarrow 0$ as $n\to\infty$.

We notice that for some $C,L>0$
\begin{equation}\label{phibeschLip}
\|\varphi \|_\infty \leq C \quad\text{and}\quad |\varphi (x) - \varphi (y)| \leq L |x-y|\, , \quad x,y\in\R^d\; .
\end{equation}

{\sl Step 1:} 

Let $v \in \De (\hat{\E}^\Sigma_\ep)$, then we shall show that for some $\tilde{C}>0$ uniformly with respect to $\ep\in (0, \ep_0]$
\begin{equation}\label{lemma3-4}
\te_\ep^\Sigma [e^{\frac{\varphi}{\ep}} v]  \leq e^{\frac{\tilde{C}}{\ep}}  \te^\Sigma_{\ep} [v]\; .
\end{equation}
By \eqref{Neumann1}, this implies $e^{\frac{\varphi}{\ep}} v \in \De (\hat{\E}^\Sigma_\ep)$.

From \eqref{phibeschLip} it follows at once that 
\begin{equation}\label{lemma3-8}
 \| e^{\frac{\varphi}{\ep}} v \|^2_{L^2(\Sigma)} \leq e^{\frac{2C}{\ep}} \|v\|^2_{L^2(\Sigma)}\; .
\end{equation}
Using the definition \eqref{VSigma} of $\hat{\VE}_\ep^\Sigma$, we have by \eqref{phibeschLip}
\begin{equation}\label{lemma3-9}
 \hat{\VE}_\ep^\Sigma [e^{\frac{\varphi}{\ep}} v] \leq e^{\frac{2C}{\ep}} \hat{\VE}_\ep^\Sigma [v] \; .
\end{equation}
It remains to analyze 
\begin{equation}\label{lemma3-10}
  \hat{\TE}_\ep^\Sigma [e^{\frac{\varphi}{\ep}} v] = \frac{1}{2} \int_\Sigma dx \int_{\Sigma'(x)} 
\Bigl(e^{\frac{\varphi(x + \ep\gamma)}{\ep}} v(x + \ep\gamma) - e^{\frac{\varphi (x)}{\ep}} v(x)\Bigr)^2 K_\ep(x, d\gamma)\; .
\end{equation}
Adding $f-f$ for $f= e^{\frac{\varphi(x+ \ep\gamma)}{\ep}} v(x)$ inside the brackets on rhs\eqref{lemma3-10} and then using $(a + b)^2 \leq 2 (a^2 + b^2)$, we get
\begin{align}\label{lemma3-3}
\text{rhs}\eqref{lemma3-10} &\leq A[v] + B[v] \qquad \text{where} \\
A[v] &:= \int_\Sigma dx \int_{\Sigma'(x)} 
e^{2\frac{\varphi(x + \ep\gamma)}{\ep}}\bigl(v(x + \ep\gamma) - v(x)\bigr)^2 K_\ep(x, d\gamma) \nonumber\\
B[v] &:=  \int_\Sigma dx \int_{\Sigma'(x)} 
v(x)^2 \bigl( e^{\frac{\varphi (x + \ep\gamma)}{\ep}} - e^{\frac{\varphi (x)}{\ep}}\bigr)^2 K_\ep(x, d\gamma)\; . \label{lemma3-3.2}
\end{align}
By \eqref{phibeschLip} we have 
\begin{equation}\label{lemma3-11}
 A[v] \leq e^{\frac{2C}{\ep}} \hat{\TE}_\ep^\Sigma [v] \; .
\end{equation}
To estimate $B[v]$, observe that
\begin{equation}\label{lemma3-12}
 | 1 - e^t| \leq e^{|t|} - 1 \, , \qquad (t\in\R)\; ,
\end{equation}
which by \eqref{phibeschLip} leads to
\begin{equation}\label{lemma3-13}
 \Bigl| e^{\frac{\varphi (x + \ep\gamma)}{\ep}} - e^{\frac{\varphi (x)}{\ep}}\Bigr| \leq 
e^{\frac{\varphi (x + \ep\gamma)}{\ep}} \Bigl( e^{\frac{1}{\ep}|\varphi (x) - \varphi (x + \ep\gamma) |} - 1\Bigr) \leq
e^{\frac{C}{\ep}} L |\gamma| e^{L |\gamma|} \; .
\end{equation}
Substituting \eqref{lemma3-13} into \eqref{lemma3-3.2} gives by Hypothesis \ref{Hypo2}(a)(i)
\begin{equation}\label{lemma3-14}
 B[v] \leq e^{\frac{2C}{\ep}} L^2 \int_\Sigma dx |v(x)|^2 \int_{\Sigma'(x)} |\gamma|^2 e^{2L|\gamma|} K_\ep (x, d\gamma) \leq 
e^{\frac{\tilde{C}}{\ep}} \|v\|^2_{L^2(\Sigma)} \; ,
\end{equation}
where $\tilde{C}$ is uniform with respect to $\ep\in (0,\ep_0]$. Inserting \eqref{lemma3-14} and
\eqref{lemma3-11} into \eqref{lemma3-3} and the result in \eqref{lemma3-10}, and combining  \eqref{lemma3-10}, \eqref{lemma3-9} and \eqref{lemma3-8} proves \eqref{lemma3-4}.\\

{\sl Step 2:} 

We prove 
\begin{equation}\label{lemma3-22} 
e^{\frac{\varphi}{\ep}} v \in \De (\E_\ep^\Sigma)\quad\text{for} \quad v\in\Ce_0^\infty (\Sigma) \subset\De (\E_\ep^\Sigma)\, . 
\end{equation}

Let $j\in\Ce_0^\infty (\R^d)$ be non-negative with $\int_{\R^d}j(x)\, dx = 1$. For $\delta>0$ we set
$j_\delta (x) := \delta^{-d}j(\frac{x}{\delta})$ and $\varphi_\delta := \varphi * j_\delta$, then $\varphi_\delta \in\Ce^\infty(\R^d)$ and 
$e^{\frac{\varphi_\delta}{\ep}} v \in \Ce_0^\infty (\Sigma)\subset \De (\E_\ep^\Sigma)$. Moreover 
\begin{equation}\label{lemma3-15}
 \| \varphi_\delta \|_\infty \leq \|\varphi \|_\infty\leq C \; , \quad \quad 
\bigl\|\varphi_\delta - \varphi\bigr\| \longrightarrow 0 \quad \text{as}\quad \delta \to 0\; 
\end{equation}
and $\varphi_\delta$ has the same Lipschitz constant $L$ as $\varphi$ (see \eqref{phibeschLip}), since
\begin{equation}\label{Lipkonst} 
|\varphi_\delta (x) - \varphi_\delta (y)| = \Bigl| \int_{\R^d} \bigl(\varphi (x-z) - \varphi (y-z)\bigr) j_\delta (z)\, dz\Bigr| 
\leq L |x-y| \; . 
\end{equation}
Assume $v\in\Ce_0^\infty (\Sigma)$, then by Step 1, $e^{\frac{\varphi}{\ep}} v \in \De (\hat{\E}_\ep^\Sigma)$. Thus it suffices to show that
\begin{equation}\label{lemma3-16}
 \te_\ep^\Sigma \Bigl[ \Bigl(e^{\frac{\varphi_\delta}{\ep}} - e^{\frac{\varphi}{\ep}}\Bigr) v\Bigr] \longrightarrow 0 \quad\text{as}
\quad \delta \to 0\; .
\end{equation}
By dominated convergence, using \eqref{lemma3-15},
\begin{equation}\label{lemma3-17}
 \Bigl\| \Bigl(e^{\frac{\varphi_\delta}{\ep}} - e^{\frac{\varphi}{\ep}}\Bigr) v \Bigr\|_{L^2(\Sigma)} 
\longrightarrow 0\qquad\text{and}\qquad 
 \hat{\VE}_\ep^\Sigma \Bigl[ \Bigl(e^{\frac{\varphi_\delta}{\ep}} - e^{\frac{\varphi}{\ep}}\Bigr) v\Bigr]\longrightarrow 0 \, , 
\quad (\delta\to 0)\; .
\end{equation}
To analyze $\hat{\TE}_\ep^\Sigma$, we set $\Phi_\delta:= e^{\frac{\varphi_\delta}{\ep}} - e^{\frac{\varphi}{\ep}}$, then
\begin{align}\label{lemma3-19}
\hat{\TE}_\ep^\Sigma & \Bigl[ \Bigl(e^{\frac{\varphi_\delta}{\ep}} - e^{\frac{\varphi}{\ep}}\Bigr) v\Bigr] = A'[v] + B'[v]\, ,
 \quad\text{where}\\
A'[v]&:= \int_\Sigma dx \int_{\Sigma'(x)} 
\Phi_\delta^2(x + \ep\gamma) \bigl(v(x + \ep\gamma) - v(x)\bigr)^2 K_\ep(x, d\gamma)\nonumber \\
B'[v] &:=  \int_\Sigma dx \int_{\Sigma'(x)} 
v(x)^2 \bigl( \Phi_\delta(x + \ep\gamma) - \Phi_\delta (x)\bigr)^2 K_\ep(x, d\gamma)\; . \nonumber
\end{align}
Since $\| \Phi_\delta\|_\infty \leq e^{\frac{C'}{\ep}}$ by \eqref{lemma3-15} uniformly with respect to $\delta>0$, $e^{\frac{2C'}{\ep}}(v(x+\ep\gamma) - v(x))^2$ is
a dominating function for the integrand of $A'[v]$, which is in $L^1(d\mu)$ for the measure $d\mu = K_\ep(x, d\gamma)dx$
in $\Sigma\times \R^d\setminus\{0\}$. Thus by the dominated convergence theorem 
\begin{equation}\label{lemma3-20} 
A'[v]\rightarrow 0\; , \qquad (\delta \to 0)
\end{equation}
 because $\|\Phi_\delta \|_\infty \rightarrow 0$  as $\delta \to 0$, by \eqref{lemma3-15}.
Similarly,
\begin{equation}\label{lemma3-21}
B'[v]\rightarrow 0\, , \qquad (\delta \to 0)
\end{equation} 
by the dominated convergence theorem (observe that using \eqref{lemma3-13} for $\varphi$ and $\varphi_\delta$, uniformly with respect to $\delta$ in view of \eqref{lemma3-15} and \eqref{Lipkonst} one finds
\[ |\Phi_\delta (x+\ep\gamma) - \Phi_\delta (x)| \leq 
\Bigl|e^{\frac{\varphi_\delta(x + \ep\gamma)}{\ep}} - e^{\frac{\varphi_\delta (x)}{\ep}}\Bigr| + \Bigl|e^{\frac{\varphi(x+\ep\gamma)}{\ep}} - e^{\frac{\varphi(x)}{\ep}}\Bigr| \leq 2 e^{\frac{C}{\ep} + L|\gamma|} L |\gamma|\, , \]
which gives a dominating function for the integrand of $B'[v]$, which in view of Hypothesis \ref{Hypo2}(a) is integrable with
respect to $d\mu$). Inserting
\eqref{lemma3-21} and \eqref{lemma3-20} into \eqref{lemma3-19} and combining the result with \eqref{lemma3-17} proves \eqref{lemma3-16} and \eqref{lemma3-22}.\\

{\sl Step 3:}

Assume $v\in\De(\E_\ep^\Sigma)$, then by Definition \ref{Def1}, there are $v_n\in\Ce_0^\infty(\Sigma)$ with 
$\te_\ep^\Sigma[v_n- v] \to 0$ as
$n\to \infty$. By Step 2, for all $n\in\N$, $e^{\frac{\varphi}{\ep}}v_n\in\De (\E_\ep^\Sigma)$, and 
\[ \te_\ep^\Sigma \bigl[ e^{\frac{\varphi}{\ep}}( v_n - v) \bigr] \rightarrow 0\, , \qquad (n\to\infty) \]
by \eqref{lemma3-4}, proving $e^{\frac{\varphi}{\ep}}v \in \De(\E_\ep^\Sigma)$.\\
\end{proof}

We will use Lemma \ref{Hepconj} and Lemma \ref{lemma3} to prove the following norm
estimate, which is a main ingredient in the proof of Theorem \ref{weig}.\\

\begin{Lem}\label{HepDchi}
Assume Hypotheses \ref{Hypo1}, \ref{Hypo2} and,
for $\Sigma\subset\R^d$ open, let $H_\ep^\Sigma$ ($\hat{H}_\ep^\Sigma$) denote the operator with Dirichlet (Neumann) boundary conditions introduced in Definition \ref{Def1}. Let $\varphi:\Sigma \ra \R$ be Lipschitz and bounded.
For $E\geq 0$ fixed, let $F_\pm : \Sigma \rightarrow [0,\infty)$ be a pair of functions such that
$F(x) := F_+(x) + F_-(x) > 0 $ and
\begin{equation}\label{F+F-bedingung}
F_+^2(x) - F_-^2(x) = V_\ep(x) + V^\varphi_{\ep,\Sigma}(x) - E\; , \qquad x\in\Sigma \; ,
\end{equation}
where $V^\varphi_{\ep,\Sigma}(x)$ is given in \eqref{Vphiep}.
Then for $u\in \De (H^\Sigma_\ep)$ (or $\De (\hat{H}_\ep^\Sigma)$) real-valued with
$Fe^{\frac{\varphi}{\ep}} u \in L^2 (\Sigma)$, we have for some $C>0$
\begin{equation}\label{FmitF-}
\| Fe^{\frac{\varphi}{\ep}} u\|^2_{L^2(\Sigma)} \leq 4 \left\| \tfrac{1}{F}e^{\frac{\varphi}{\ep}}(H_\ep^\Sigma - E)
u\right\|^2_{L^2(\Sigma)} + 8 \|F_- e^{\frac{\varphi}{\ep}}u\|^2_{L^2(\Sigma)}  \, .
\end{equation}
\end{Lem}

\begin{proof}
First observe that for $v:= e^{\frac{\varphi}{\ep}} u$ 
\begin{equation}\label{FF-F+}
\| Fv\|^2_{L^2(\Sigma)} \leq 2\left( \|F_+v\|^2_{L^2(\Sigma)} + \| F_-v\|^2_{L^2(\Sigma)} \right)  =
2\left( \|F_+v\|^2_{L^2(\Sigma)} - \| F_-v\|^2_{L^2(\Sigma)} \right) + 4 \| F_-v\|^2_{L^2(\Sigma)}\; .
\end{equation}
By \eqref{F+F-bedingung} one has
\begin{equation}\label{eikoF}
  \| F_+v\|^2_{L^2(\Sigma)} - \| F_-v\|^2_{L^2(\Sigma)} =
 \skpSi{(V_\ep + V_{\ep,\Sigma}^\varphi - E)v}{v} \, .
\end{equation}
Since $v\in \De (\E_\ep^\Sigma)$ (or $\De(\hat{\E}_\ep^\Sigma)$) by Lemma \ref{lemma3}, it follows at once 
from Lemma \ref{Hepconj} that
\begin{equation}\label{HepundVvarphi}
\skpSi{(V_\ep + V_{\ep,\Sigma}^\varphi - E)v}{v} \leq
\E_\ep^\Sigma\Bigl(e^{-\frac{\varphi}{\ep}} v, e^{\frac{\varphi}{\ep}}v \Bigr) - E\|v\|_{L^2(\Sigma)}^2
 \, .
\end{equation}
\eqref{eikoF} and \eqref{HepundVvarphi} yield by use of the Cauchy-Schwarz inequality, since $u\in\De (H_\ep^\Sigma)$,
\begin{eqnarray}\label{schwarzbinom}
2\left( \|F_+v\|^2_{L^2(\Sigma)} - \| F_-v\|^2_{L^2(\Sigma)} \right) &\leq& 2
\skpSi{\left( e^{\frac{\varphi}{\ep}}(H^\Sigma_\ep - E)\right) u}{v}
 \\
&\leq& 2\sqrt{2}
\left\|\tfrac{1}{F}\left( e^{\frac{\varphi}{\ep}}(H^\Sigma_\ep - E)\right) u
\right\|_{L^2(\Sigma)}
\frac{1}{\sqrt{2}}\| Fv\|_{L^2(\Sigma)} \nonumber\\
&\leq&  2 \left\|\tfrac{1}{F}\left( e^{\frac{\varphi}{\ep}}(H^\Sigma_\ep - E)
\right) u\right\|^2_{L^2(\Sigma)} +\frac{1}{2} \| Fv\|^2_{L^2(\Sigma)}\, .\nonumber
\end{eqnarray}
Inserting \eqref{schwarzbinom} into \eqref{FF-F+} we get
\[
\| Fv\|^2_{L^2(\Sigma)} \leq  2 \left\|\tfrac{1}{F}
\left( e^{\frac{\varphi}{\ep}}(H^\Sigma_\ep - E)
\right) u\right\|^2_{L^2(\Sigma)} +\frac{1}{2} \| Fv\|^2_{L^2(\Sigma)}  +
4 \| F_-v\|^2_{L^2(\Sigma)} \; , \]
which by  definition of $v$ gives \eqref{FmitF-}.\\
\end{proof}

\begin{Lem}\label{propt}
Assume Hypotheses \ref{Hypo1} and \ref{Hypo2}. 
\ben
\item For any $x\in\R^d$, the 
function $L_x: \R^d\ni\xi\mapsto \tilde{t}_0(x, \xi)$ is even and hyperconvex, i.e. $D^2L_x|_{\xi_0} \geq \alpha >0$, uniformly  in $\xi_0$.
\item At $\xi=0$, for fixed $x\in\R^d$,  the function $\tilde{t}_0$ has an expansion
\begin{equation}\label{kinen}
0 \leq \tilde{t}_0(x,\xi) - \skp{\xi}{B(x)\xi} = O\left(|\xi|^4\right)\qquad\text{as}\;\; |\xi|\to 0\, ,
\end{equation}
where $B:\R^d\rightarrow\mathcal{M}(d\times d,\R)$ is positive
definite, symmetric and bounded.
\een
\end{Lem}

\begin{proof}
(a): 
By Hypothesis \ref{Hypo2}(a)(ii) there exists $c_x>0$ such that  for all $\xi_0, v\in\R^d$
\[
 \skp{v}{D^2L_x|_{\xi_0} v} = \int_{\R^d\setminus\{0\}} (\gamma \cdot v)^2 \cosh (\gamma\cdot\xi_0) K^{(0)}(x, d\gamma) 
\geq \int_{\R^d\setminus\{0\}} (\gamma \cdot v)^2  K^{(0)}(x, d\gamma) \geq c_x \|v\|^2\; .
\]
(b):
Since by Taylor expansion at $\xi=0$
\[ \cosh (\gamma\cdot\xi) - \bigl( 1 + \tfrac{1}{2}(\gamma\cdot\xi)^2\bigr) \leq (\gamma\cdot\xi)^4\, 
\frac{\sinh (\gamma\cdot\xi)}{(\gamma\cdot\xi)}\; , \]
one gets from \eqref{tildetdef} and Hypotheses \ref{Hypo1} and \ref{Hypo2}
\[
0 \leq\bigl|\tilde{t}_0 (x, \xi)  - \skp{\xi}{B(x)\xi}\bigr|  \leq \int_{\R^d\setminus\{0\}}  (\gamma\cdot\xi)^4\, 
\frac{\sinh (\gamma\cdot\xi)}{(\gamma\cdot\xi)} K^{(0)}(x, d\gamma)  =  O\left(|\xi|^4\right) \; ,
\]
as $|\xi| \to 0$, where the symmetric $d\times d$-matrix $B=(B_{\mu\nu})$ is given by
\[
  B_{\nu\mu}(x) = \frac{1}{2}\int_{\R^d} \gamma_\nu\gamma_\mu K^{(0)}(x, d\gamma) 
\qquad
\mbox{for}\quad \mu,\nu\in\{1,\ldots,d\}\; , x\in\R^d\; .
\]
By Hypothesis \ref{Hypo2}(a)(ii), $B$ is strictly positive definite, by Hypothesis \ref{Hypo2}(a)(i), $B$ is bounded.\\
\end{proof}

\begin{Lem}\label{d}
 Assume Hypotheses \ref{Hypo1}, \ref{Hypo2} and \ref{Hypo3}, then
\ben
\item $d(x) = \frac{1}{2}\skp{x-x_j}{D^2d|_{x_j} (x-x_j)} + o(|x-x_j|^2)$ as $|x-x_j|\to 0$,  and $D^2d|_{x_j}$ is positive definite.
\item $\nabla d(x) = O(|x-x_j|)$ as $|x-x_j|\to 0$.
\een
\end{Lem}

\begin{proof}
(b): For $|x-x_j|$ sufficiently small, the eikonal equation \eqref{eikonal} holds. Thus by Lemma \ref{propt} (b), we have,
with $B(x)$ positive definite and bounded,
\begin{align}\label{d-1}
 V_0(x) &= \tilde{t}(x, \nabla d(x)) = \skp{ \nabla d(x)}{ B(x) \nabla d(x)} + O (|\nabla d(x)|^4)  \\
&\geq C |\nabla d(x)|^2  \,  \label{d-3}\; .
\end{align}
\eqref{d-3} proves (b), since $V_0(x) = O(|x-x_j|^2)$ 
(Hypothesis  \ref{Hypo2}(b)).

(a): Since $d\in\Ce^2(\O)$, $d(x_j)=0$ and  $\nabla d(x_j) = 0$ (use (b)), Taylor expansion gives
\[
 d(x) = \frac{1}{2}\skp{x-x_j}{D^2d|_{x_j} (x-x_j)} + o(|x-x_j|^2)\qquad\text{as}\quad |x-x_j|\to 0\; .
\]
Since $d(x) \geq 0$, the matrix $D^2d|_{x_j}$ is non-negative. We shall now assume that $0$ is an eigenvalue of
$D^2d|_{x_j}$ with eigenspace $\mathcal{N}\subset \R^d$ and derive a contradiction.

By the mean value theorem and the continuity of $D^2d|_{x}$
\[ \nabla d(x) = \int_0^1 D^2d|_{x_j+t(x-x_j)} (x-x_j) \, dt = D^2d|_{x_j} (x-x_j) + o(|x-x_j|)\qquad (|x-x_j|\to 0) .\]
Thus
\[ \nabla d(x) = o(|x-x_j|) \qquad (|x-x_j|\to 0, (x - x_j) \in\mathcal{N}) \; . \]
By \eqref{d-1} this gives $V_0(x) = o(|x-x_j|^2)$   as $x-x_j\to 0$ in $\mathcal{N}$, which contradicts $D^2V(x_j)>0$ (Hypothesis \ref{Hypo2}(b)). Thus $D^2d|_{x_j}$ is positive definite.
\end{proof}

\begin{Lem}\label{Philem}
Assume Hypotheses \ref{Hypo1}, \ref{Hypo2} and \ref{Hypo3} and let $\chi\in{\Ce}^\infty(\R_+,[0,1])$ such that $\chi (r)=0$ for $r\leq
\frac{1}{2}$ and $\chi (r) =1$ for $r\geq 1$. In addition we assume that
$0\leq \chi'(r) \leq \frac{2}{\log 2}$.
For $B>0$ we define
$g: \overline{\Sigma} \ra [0,1]$ by
\begin{equation}\label{defg}
g(x):= \chi\left(\frac{d(x)}{B\ep}\right)\; ,\qquad x\in \overline{\Sigma}
\end{equation}
and set
\begin{equation}\label{DefPhix}
\Phi (x) := d(x) - \frac{B\ep}{2}\ln \left(\frac{B}{2}\right) -
g(x)\frac{B\ep}{2} \ln \left(\frac{2d(x)}{B\ep}\right)\; ,\qquad x\in \overline{\Sigma}\, .
\end{equation}
Then $\Phi\in\Ce^2(\overline{\Omega})$ and there exists a constant $C>0$ such that for all $\ep\in(0,\ep_0]$
\begin{equation}\label{deltaphibound}
\left|\partial_\nu\partial_\mu \Phi(x)\right|\leq C \, , \qquad x\in\Sigma\, , \mu,\nu\in \{1,\ldots d\}\; .
\end{equation}
Furthermore, for any $B>0$ there is $C'>0$ such that
\begin{equation}\label{ePhied}
e^{\frac{d(x)}{\ep}}\frac{1}{C'}\left(1+\frac{d(x)}{\ep}\right)^{-\frac{B}{2}}\leq
e^{\frac{\Phi(x)}{\ep}} \leq
e^{\frac{d(x)}{\ep}} C'\left(1+\frac{d(x)}{\ep}\right)^{-\frac{B}{2}}\; .
\end{equation}
\end{Lem}

\begin{proof}
Using the estimates of Lemma \ref{d}, the proof follows word by word the proof of Lemma 3.3 in \cite{kleinro}.
\end{proof}

\begin{Lem}\label{tnullVnull}
Let $j\in\Ce_0^\infty (\R^d)$ be non-negative with $\int_{\R^d}j(x)\, dx = 1$ and 
$\supp j\subset B_1 (0):= \{x\in\R^d\, |\, |x|<1\}$. 
For $\delta>0$ we introduce the Friedrichs mollifier
$j_\delta (x) := \delta^{-d}j(\frac{x}{\delta})$.
Under the assumptions of Theorem \ref{Theorem2}, setting $d_\delta :=d * j_\delta$,
we have, locally uniformly in $x\in\R^d$,
\begin{equation}\label{tnullVnullglg}
 V_0(x) \geq \tilde{t}_0(x, \nabla d_\delta (x)) + o(1) \qquad (\delta \to 0)\; .
\end{equation}
\end{Lem}

We emphasize that $\nabla d_\delta $ does {\em not} converge to $\nabla d $ in $|| . ||_\infty$.
The estimate \eqref{tnullVnullglg} compensates. This is crucial to obtain the positivity needed in our Agmon estimate.

\begin{proof}
 First observe that by \eqref{K0stetig}, \eqref{abfallagamma1} and \eqref{abfallagamma2}
\begin{equation}\label{tnullVnull1}
 \tilde{t}_0(x-y, \xi) - \tilde{t}_0(x,\xi) = \int_{R^d\setminus\{0\}} \bigl( \cosh \gamma\cdot\xi - 1\bigr) 
\bigl( K^{(0)}(x-y, d\gamma) - K^{(0)}(x, d\gamma)\bigr) = o(1)
\end{equation}
as $|y|\to 0$ locally uniformly in $(x,\xi)\in\R^{2d}$ (since $|\cosh \gamma\cdot\xi - 1| \leq C |\gamma|^2 e^{C|\gamma|}$).

We remark that 
\begin{equation}\label{tnullVnull2}
\nabla d_\delta (x) = \int_{\R^d} \nabla d(x-y) j_\delta (y) \, dy = \mathbb{E}_\delta \bigl[\nabla d(x - \,.\,)\bigr] \, ,
\end{equation}
 where $\mathbb{E}_\delta$ denotes expectation with respect to the probability measure 
$d\mu_\delta (y) = j_\delta(y)\, dy$ (supported in the ball $B_\delta(0)$). Recall the multidimensional Jensen
inequality (see e.g. \cite{dud})
\begin{equation}\label{Jensen}
 \mathbb{E}\bigl[ f(X)\bigr] \geq f\bigl(\mathbb{E}[X]\bigr)
\end{equation}
for any convex function $f: \R^d \rightarrow \R$ and random variable $X$ with values in $\R^d$. Choosing 
$X(\,.\,)=\nabla d(x-\,.\,)$ and using the convexity of $\tilde{t}_0(x,\,.\,)$ (see Lemma \ref{propt}), we get by
\eqref{tnullVnull2} and \eqref{Jensen}
\begin{align}
\tilde{t}_0(x, \nabla d_\delta (x) ) &\leq \int_{\R^d} \tilde{t}_0(x, \nabla d(x - y) ) d\mu_\delta (y) \nonumber\\
&=  \int_{\R^d} \tilde{t}_0(x-y, \nabla d(x - y) ) d\mu_\delta (y) + o(1) \qquad (\delta \to 0)\, ,\label{tnullVnull3}
 \end{align}
where the last equality follows from \eqref{tnullVnull1} and $\supp j_\delta \subset B_\delta (0)$.
Thus, by \eqref{tnullVnull3} and the eikonal inequality \eqref{eikonalun}
\[
 \tilde{t}_0(x, \nabla d_\delta (x) )  \leq \int_{\R^d}V_0(x-y) d\mu_\delta (y) + o(1) \leq V_0(x) + o(1) \qquad (\delta \to 0)
\]
\end{proof}

\section{Proof of Theorem 1.5 and 1.6}\label{proofweig}

\begin{proof}[Proof of Theorem 1.6]

We partly follow the ideas in the proof of Theorem 1.7 in \cite{kleinro}.\\

{\sl Proof of (b):}\\

For $\Sigma'(x)$ given in \eqref{Sigma'}, let
\[
\tilde{t}_0^\Sigma(x,\xi) :=
\int_{\Sigma'(x)} \left(\cosh \left(\gamma\cdot\xi\right) - 1\right) K^{(0)}(x, d\gamma) , 
\qquad (x,\xi)\in\Sigma\times\R^d\, ,
\]
then by the positivity of the integrand 
\begin{equation}\label{tkleiner}
\tilde{t}_0^\Sigma(x,\xi) \leq \tilde{t}_0(x,\xi)\; .
\end{equation}
For any $B>0$ we choose $\ep_B>0$ such that $d^{-1}([0,\ep_B B))\subset U$, then by Hypothesis \ref{Hypo3} for all $\ep<\ep_B$
\begin{equation}\label{condB}
V_0(x) - \tilde{t}_0(x,\nabla d(x)) = 0 \, ,
\qquad x\in\Sigma\cap d^{-1}([0,B\ep)) \, .
\end{equation}
Let $\Phi$ be given in \eqref{DefPhix}, then by \eqref{defg} 
\begin{equation}\label{gradPhi}
\nabla \Phi(x) = \nabla d(x) (1-f_1(x) - f_2(x)) \; ,
\end{equation}
where
\[ 
f_1(x) := \frac{B\ep}{2d(x)}\chi\left(\frac{d(x)}{B\ep}\right)\quad\text{and}\quad
f_2(x) := \frac{1}{2}
\chi'\left(\frac{d(x)}{B\ep}\right)\,\log \left(\frac{2d(x)}{B\ep}\right).
\]
Choose $\eta>0$ such that $\tilde{K}:=d^{-1}([0, D+2\eta])\subset \O$ and let
$\hat{\chi}, \tilde{\chi}\in\Ce^\infty (\R_+, [0,1])$ be monotone with 
\[
\tilde{\chi}(x)= \begin{cases} 0\, , \quad x\leq D+\eta\\ 1\, , \quad x \geq D+ 2\eta\end{cases}\qquad 
\hat{\chi}(x)= \begin{cases} 0\, , \quad x\leq D  \\ 1\, , \quad x \geq D+\eta \end{cases}\; .
\]
Then we define 
\begin{equation}\label{tildehatg}
  \tilde{g}(x):= \tilde{\chi}(d(x))\qquad\text{and}\qquad \hat{g}(x):= \hat{\chi}(d(x))
\end{equation}
and we set for $\delta>0$
\[
 \Phi_{\alpha, \delta} (x) = (1-\hat{g}(x)) \Phi(x) + \hat{g}(x)  \bigl( 1- \tfrac{\alpha}{2}\bigr)
 \bigl((1-\tilde{g}(x)) d(x) + \tilde{g}(x) d_\delta (x)\bigr)\; ,
\]
where $d_\delta= d*j_\delta$ is defined in Lemma \ref{tnullVnull}. Then $\Phi_{\alpha, \delta}\in \Ce^2(\overline{\Sigma})$ for any $\delta>0$.\\

{\sl Step 1:} We show that there is $\delta(\alpha)$ such that for any $\delta<\delta(\alpha)$ the function $\Phi_\alpha := \Phi_{\alpha, \delta}$ satisfies
\eqref{abinK} and \eqref{abohneK}.\\

Clearly, $\Phi_{\alpha, \delta}$ satisfies \eqref{abinK} for all $\delta>0$ in view of \eqref{ePhied}, since $\Phi_{\alpha, \delta} (x) = \Phi (x)$ for $x\in K$.\\
Now, by \eqref{eikonal}, for $x\in\Sigma\setminus K$
\begin{equation}\label{thm2-2}
  \Phi_{\alpha, \delta} (x) = d(x) - \hat{g}(x)\tfrac{\alpha}{2} d(x) - (1-\hat{g})(x) \Bigl( \frac{B\ep}{2} 
\ln \Bigl(\frac{d(x)}{\ep}\Bigr)\Bigr) + 
  \tilde{g}(x) \bigl(1-\tfrac{\alpha}{2}\bigr) \bigl( d_\delta - d\bigr)(x)
\end{equation}
Choosing $B\geq 2$, all logarithms in \eqref{thm2-2} are positive 
(using $\frac{2 d(x)}{B\ep} \geq 1$ on the support of $g$). Since $\|d_\delta - d\|_\infty \to 0$ as $\delta \to 0$ and
using that for some $C$, by Hypothesis \ref{Hypo3},
\begin{equation}\label{thm2-5}
 \inf \{d(x)\,|\, x\in\Sigma\setminus K\} \geq C >0\, ,
\end{equation}
it follows that there is a 
$\delta(\alpha)$ such that for all $\delta < \delta(\alpha)$
\begin{equation}\label{thm2-3}
 \bigl(1-\tfrac{\alpha}{2}\bigr) \bigl|d_{\delta}(x) - d(x)\bigr| \leq \tfrac{\alpha}{2} d(x)\, , \quad x\in \Sigma\setminus K\, ,
\end{equation}
proving the upper bound in \eqref{abohneK} for $\Phi_\alpha$.

Now observe that there is an $\ep_\alpha>0$ such that for all $\ep\in (0,\ep_\alpha)$
\begin{equation}\label{thm2-4}
 \frac{B\ep}{2} \ln \Bigl(\frac{d(x)}{\ep}\Bigr) \leq \frac{\alpha}{4} d(x)\, , \qquad x \in \Sigma\setminus K \; .
\end{equation}
This follows from the fact that lhs\eqref{thm2-4}$=o(1)$ as $\ep \to 0$ uniformly in $x$ together with \eqref{thm2-5}.
Inserting \eqref{thm2-4} and \eqref{thm2-3} into \eqref{thm2-2} proves the lower bound of \eqref{abohneK}. \\

{\sl Step 2:} We shall show that there are constants $\alpha_0, C_0, C_1>0$ independent of $B$ and $E$ and $\ep_\alpha, \delta(\alpha)>0$ such that
for all $\delta<\delta(\alpha), \, \ep<\ep_\alpha$ and for any fixed $\alpha\in (0,\alpha_0]$
\begin{equation}\label{Vnulltab}
V_0(x) - \tilde{t}_0^\Sigma(x, \nabla\Phi_{\alpha,\delta}(x)) \geq \begin{cases} 0\, ,&\, x\in\Sigma\cap d^{-1}([0, B\ep)) \\
\frac{B}{C_0}\ep \, ,&\, x\in\Sigma\cap d^{-1}([B\ep, D+\eta))\\
C_1 \, , &\, x\in \Sigma \cap d^{-1}([D+\eta, \infty))
\end{cases}
\end{equation}

{\sl Case 1:} $d(x) \leq \frac{B\ep}{2}$\\
Since $\Phi_{\alpha, \delta}(x) = d(x) - \frac{B\ep}{2}\ln\bigl(\frac{B}{2}\bigr)$ and the eikonal equation
\eqref{eikonal} holds, we get
\[
V_0(x) - \tilde{t}_0(x, \nabla \Phi_{\alpha,\delta}(x)) = V_0(x) - \tilde{t}_0(x,\nabla d(x)) = 0 \, ,\qquad
x\in\Sigma\cap d^{-1}([0, \tfrac{B\ep}{2}]) \; .
\]
which by \eqref{tkleiner} leads  to \eqref{Vnulltab}.\\

{\sl Case 2:} $\frac{B\ep}{2}<d(x)<B\ep$\\
Here $\Phi_{\alpha,\delta}(x) = \Phi(x)$.  
Since $1<\frac{2d(x)}{B\ep}<2$, $f_1$ and $f_2$ in \eqref{gradPhi} are non-negative.
In addition $0\leq f_j(x) \leq 1, j=1,2$ (use assumption
$\chi'(r)\leq \frac{2}{\log 2}$). Therefore
\begin{equation} \label{f}
|1-f_1(x) - f_2(x)| =: |\lambda(x)| \leq 1.
\end{equation}
By Lemma \ref{propt}, $\tilde{t}_0(x,\xi)$ is convex with respect to $\xi$, therefore
\begin{equation}\label{convexun}
\tilde{t}_0(x,\lambda \xi + (1-\lambda)\eta) \leq \lambda \tilde{t}_0(x,\xi) +
(1-\lambda) \tilde{t}_0(x, \eta)\quad\mbox{for}\quad 0\leq \lambda\leq 1,\; \xi,\eta\in \R^d\; .
\end{equation}
and, since $\tilde{t}_0(x,0)=0$ and $\tilde{t}_0(x,\xi) = \tilde{t}_0(x,-\xi)$, it follows by choosing $\eta=0$ that
\begin{equation}\label{convexun2}
 \tilde{t}_0(x, \lambda \xi) \leq |\lambda| \tilde{t}_0 (x, \xi)\, , \qquad \text{for}\quad \lambda\in\R, |\lambda|\leq 1, 
\xi\in\R^d, x\in\Sigma\; .
\end{equation}
Combining \eqref{f}, \eqref{convexun2} and \eqref{tkleiner}
it follows that
\begin{equation}\label{Vnulltlambda}
V_0(x) - \tilde{t}_0^\Sigma(x, \nabla \Phi_{\alpha, \delta} (x))  \geq
V_0(x) - |\lambda(x)| \tilde{t}_0(x, \nabla d(x)) \geq  V_0 (1-|\lambda (x)|)\; ,
\end{equation}
where for the second step we used \eqref{condB}.
Since $|\lambda(x)|\leq 1$ and $V_0\geq 0$, \eqref{Vnulltlambda} gives \eqref{Vnulltab}.\\

{\sl Case 3:} $B\ep\leq d(x) < D$\\
In this region, we have $\Phi_{\alpha,\delta}(x) =\Phi (x)= d(x) - \frac{B\ep}{2}\ln\bigl(\frac{d(x)}{\ep}\bigr)$, thus 
\begin{equation}\label{wei1}
\nabla \Phi_{\alpha,\delta} (x)  = \nabla d(x) \left(1-\frac{B\ep}{2d(x)}\right) \; .
\end{equation}

Since $\frac{1}{2}\leq (1-\frac{B\ep}{2d(x)}) < 1$,
by \eqref{wei1} and \eqref{convexun2} we get the estimate
\begin{align}
V_0(x) - \tilde{t}_0(x, \nabla \Phi_{\alpha,\delta} (x)) &\geq V_0(x) - \left(1-\frac{B\ep}{2d(x)}\right)
\tilde{t}_0(x, \nabla d(x)) \nonumber \\
&\geq 
V_0(x) \frac{B\ep}{2d(x)} \, ,\label{Vnullcase2}
\end{align}
where for the second estimate we used that by Hypothesis \ref{Hypo3} the eikonal inequality $\tilde{t}_0(x, \nabla d(x)) \leq
V_0(x)$ holds. 
We now claim that there exists a constant $C_0>0$ such that
\begin{equation}\label{Vnulld}
\frac{V_0(x)}{2d(x)} \geq C^{-1}_0 \, ,\qquad
x\in\Sigma\cap d^{-1}([ B\ep,\infty)) \; .
\end{equation}
Then, combining \eqref{tkleiner}, \eqref{Vnullcase2} and \eqref{Vnulld},
we finally get \eqref{Vnulltab}.

To see \eqref{Vnulld}, we split the region $W = \Sigma\cap d^{-1}([ B\ep,\infty))$ into two parts.
Clearly, for any $\delta>0$, \eqref{Vnulld} holds for $x\in W \cap \{|x-x_j|>\delta\}$ (since
$\Sigma$ is bounded, $d\in\Ce^2(\overline{\Sigma})$ and $V_0(x)\geq C>0$ for $|x-x_j|>\delta$ by Hypothesis \ref{Hypo2},(b)).

To discuss the region $W \cap \{|x-x_j|\leq \delta\}$, we remark that for some $C>0$ by Hypothesis \ref{Hypo2},(b) 
$V_0(x) \geq C |x-x_j|^2$ if $|x-x_j|\leq \delta$. Thus it suffices to show that for some $\tilde{C}>0$
\[
 d(x) \leq C |x-x_j|^2\, , \qquad |x-x_j|\leq \delta\; .
\]
This follows from Lemma \ref{d}(a).\\

{\sl Case 4:} $D \leq d(x) < D + \eta$\\
Since $\Phi_{\alpha,\delta}(x) = (1-\hat{g}(x)) \Phi(x) + \hat{g}(x)\bigl(1-\frac{\alpha}{2}\bigr)d(x)$ and $\nabla \Phi(x)$ is given by \eqref{wei1} in this region, we have
\begin{align}\label{case5-1}
 \nabla \Phi_{\alpha,\delta}(x) &= \nabla d(x) \Big[ \Bigl( 1 - \frac{B\ep}{2 d(x)}\Bigr) (1 - \hat{g}(x)) -
\hat{\chi}'(d(x)) \Bigl(\Bigl(d(x)- \frac{B\ep}{2}\ln\Bigl(\frac{d(x)}{\ep}\Bigr)\Bigr) + 
 \nonumber\\ 
&\hspace{1cm}+
 \hat{\chi}'(d(x)) \bigl(1-\tfrac{\alpha}{2}\bigr) d(x) + \hat{g}(x) \bigl(1-\tfrac{\alpha}{2}\bigr)\Big]\nonumber \\
&= \lambda \nabla d(x) \; ,
\end{align}
where
\[
\lambda= 1 + h_\alpha(x) - (1-\hat{g}(x))\frac{B\ep}{2d(x)} - \hat{g}(x)\frac{\alpha}{2}, \qquad
h_\alpha(x) := \hat{\chi}'(d(x)) \Bigl(-\frac{\alpha}{2} d(x) + \frac{B\ep}{2} \ln \Bigl(\frac{d(x)}{\ep}\Bigr)\Bigr).
\]
Since $\hat{\chi}'(y) \geq 0$ it follows from the upper bound in \eqref{thm2-4} that $h_\alpha\leq 0$, proving for $\alpha$
sufficiently small
\begin{equation} \label{newl}
 0 \leq    \lambda \leq 1 - t, \qquad t=t(x,\alpha, \ep) =(1-\hat{g}(x))\frac{B\ep}{2d(x)} + \hat{g}(x)\frac{\alpha}{2}\; .
\end{equation}
Combining   \eqref{tkleiner},  \eqref{convexun2},  \eqref{case5-1}  and\eqref{newl} gives, for all $\ep\leq \ep_\alpha$ sufficiently small
\begin{align*}
V_0(x) - \tilde{t}^\Sigma_0(x, \nabla\Phi_{\alpha,\delta}) 
&\geq V_0(x) t(x,\alpha,\ep)
\geq \frac{B}{C_0}\ep\,,
\end{align*}
where we used \eqref{eikonalun} and, for the last estimate,  \eqref{Vnulld}.\\ 

{\sl Case 5:} $D + \eta \leq d(x) < D+2\eta$\\
We have $\Phi_{\alpha,\delta}(x) = \bigl(1-\frac{\alpha}{2}\bigr) ((1- \tilde{g}(x)) d(x) + \tilde{g}(x)d_\delta(x)$ and thus
\begin{align}
\nabla\Phi_{\alpha,\delta}(x) &= \bigl(1-\tfrac{\alpha}{2}\bigr) \bigl[ (1-\tilde{g}(x)) \nabla d(x) + \tilde{\chi}'(d(x))
(d_\delta(x) - d(x)) \nabla d(x) + \tilde{g}(x) \nabla d_\delta(x)\bigr]\nonumber \\
&= \bigl(1-\tfrac{\alpha}{2}\bigr) \bigl[ (1-\tilde{g}(x)) \nabla d(x) + \tilde{g}(x) \nabla d_\delta(x)\bigr] +
\tfrac{\alpha}{2}\tfrac{2}{\alpha}f_\delta(x)\nabla d(x)\, ,\label{case6-1}
\end{align}
where we set 
\[
f_\delta (x) := \tilde{\chi}'(d(x))(d_\delta(x) - d(x)) \quad\text{and thus}\quad 
|f_\delta(x)| = o(1)\quad (\delta\to 0) \;.
\]
Using \eqref{convexun} twice, we get by \eqref{case6-1}
\[
\tilde{t}_0(x, \nabla \Phi_{\alpha,\delta}(x)) \leq \bigl(1-\tfrac{\alpha}{2}\bigr)
\Bigl[(1-\tilde{g}(x)) \tilde{t}_0(x, \nabla d(x)) + \tilde{g}(x) \tilde{t}_0(x, \nabla d_\delta (x))\Bigr] 
+ \tfrac{\alpha}{2} \tilde{t}_0\bigl(x, \tfrac{2}{\alpha} f_\delta(x)\nabla d(x)\bigr)\; .
\]
Combining Lemma \ref{tnullVnull} 
with \eqref{eikonalun} 
yields, as  $\delta  \to 0,$
\begin{equation}\label{case6-4}
V_0(x) - \tilde{t}_0 \bigl(x, \nabla\Phi_{\alpha,\delta}(x) \bigr)
  \geq V_0(x) \Bigl[ \frac{\alpha}{2} + o(1)\Bigr] \geq C_1\; ,
\end{equation}
since $V_0(x)\geq C>0$ in this region.
Combining \eqref{tkleiner} with  \eqref{case6-4} gives \eqref{Vnulltab}.\\

{\sl Case 6:} $d(x) \geq D+2\eta$\\
Since $\Phi_{\alpha,\delta}(x) = \bigl(1-\frac{\alpha}{2}\bigr)d_\delta(x)$, we have by \eqref{convexun2}
\begin{equation}\label{case4-1}
 \tilde{t}_0\bigl(x, \nabla \Phi_{\alpha,\delta}(x)\bigr) 
\leq  \Bigl(1-\frac{\alpha}{2}\Bigr) \tilde{t}_0 \bigl(x, \nabla d_\delta (x)\bigr)
\end{equation}
Combining Lemma \ref{tnullVnull} with  \eqref{case4-1} gives \eqref{case6-4}, as in {\sl Case 5}.\\

{\sl Step 3:} We shall show
\begin{equation}\label{VepundVPhi}
V_\ep(x) + V^{\Phi_\alpha}(x) \geq \begin{cases}  - C_2\,\ep & \qquad\mbox{for}
\quad x\in\Sigma\cap d^{-1}([0, B\ep]) \\
\left(\frac{B}{C_0}-C_3\right)\ep &
\qquad\mbox{for}\quad  x\in\Sigma\cap d^{-1}([B\ep, D+\eta)) \\
C_4 & \qquad\mbox{for}\quad  x\in\Sigma\cap d^{-1}([D+\eta, \infty))
\end{cases}
\end{equation}
for some $C_2,C_3, C_4>0$ independent of $B$ and $E$, where $V^{\Phi_\alpha}:=V_{\ep,\Sigma}^{\Phi_\alpha}$ is defined in \eqref{Vphiep} and $\Phi_\alpha = \Phi_{\alpha,\delta}$ for any $\delta<\delta(\alpha)$.\\

We write
\begin{equation}\label{wei5}
V_\ep(x) + V^{\Phi_\alpha}(x) = \left(V_\ep(x) - V_0(x)\right) + \left(V^{\Phi_\alpha}(x) + \tilde{t}_0^\Sigma(x, \nabla\Phi_\alpha(x))\right)
 + \left(V_0(x) -
\tilde{t}_0^\Sigma(x, \nabla\Phi_\alpha(x))\right)
\end{equation}
By Hypothesis \ref{Hypo1} and since $\Sigma$ is bounded, there exists a constant $C_1>0$ such that
\begin{equation}\label{Vephalbbe}
V_\ep(x) - V_0(x) \geq - C_1 \ep \, , \qquad x\in\Sigma\; .
\end{equation}
We shall show that
\begin{equation}\label{Vphitsigma}
\left| V^{\Phi_\alpha} (x) + \tilde{t}_0^\Sigma(x, \nabla \Phi_\alpha(x))\right| \leq \ep C_2 \; .
\end{equation}
Then inserting \eqref{Vphitsigma}, \eqref{Vephalbbe} and \eqref{Vnulltab} into \eqref{wei5} proves \eqref{VepundVPhi}.
Setting (see \eqref{Vphiep})
\[ V_{0}^{\Phi_\alpha}(x) :=
\int_{\Sigma'(x)}\left[ 1 -  \cosh \left(F_\alpha(x)\right)\right] K^{(0)}(x, d\gamma), \quad  
F_\alpha(x)=F_\alpha(x,\gamma,\ep)= \tfrac{1}{\ep} (\Phi_\alpha(x)- \Phi_\alpha(x+\ep\gamma))
\]
we write
\[ V^{\Phi_\alpha}(x) + \tilde{t}_0^\Sigma(x, \nabla\Phi_\alpha(x))= \left( V^{\Phi_\alpha}(x) - V_0^{\Phi_\alpha}(x)\right) +
\left( V_0^{\Phi_\alpha} (x) + \tilde{t}_0^\Sigma(x, \nabla\Phi_\alpha(x)) \right) =: D_1(x) + D_2(x) \]
and analyze the two summands on the right hand side separately.
Since $\Phi_\alpha\in\Ce^2(\overline{\Sigma})$, it follows from Hypotheses
\ref{Hypo1} and \ref{Hypo2}(a), using \eqref{lemma2-4}, that for some $\tilde{C}>0$
\begin{equation}\label{D1}
\left|D_1(x)\right| = \Bigl|\int_{\Sigma'(x)}\left[ 1 -  \cosh \left(F_\alpha(x)\right)\right] \left(\ep K^{(1)} + R^{(2)}_\ep\right)(x, d\gamma) \Bigr| \leq \tilde{C} \ep\, .
\end{equation}
uniformly with respect to $x$.
We have for $x\in\Sigma$
\begin{equation}\label{Vphiminust}
\left|D_2(x)\right|  \leq \int_{\Sigma'(x)}
\Bigl|\cosh\bigl(\gamma\nabla\Phi_\alpha(x)\bigr) - 
\cosh\bigl(F_\alpha(x)\bigr)
\Bigr| K^{(0)}(x,d\gamma) \;.
\end{equation}
By the mean value theorem for $\cosh z$ and since $|\sinh x|\leq e^{|x|}$
\begin{multline}\label{mittel}
\Bigl|\cosh\bigl(F_\alpha(x)\bigr)-\cosh\bigl(\gamma\nabla\Phi_\alpha(x)\bigr)\Bigr|
\leq \sup_{t\in[0,1]} \exp\bigl(\bigl| F_\alpha(x)t +\gamma\nabla\Phi_\alpha(x) (1-t)\bigr|\bigr)
 \bigl| F_\alpha(x)  + \gamma\nabla\Phi_\alpha(x) \bigr| \; .
\end{multline}
Since $\Phi_\alpha\in\Ce^2(\overline{\Sigma})$ there exist constants
$c_1,c_2>0$ such that
 \begin{equation}\label{lip}
 |F_\alpha(x)| \leq c_1 |\gamma| \quad\mbox{and}\quad |\gamma\nabla\Phi_\alpha(x)|\leq
 c_2 |\gamma| \; , \qquad x\in\Sigma,\, \gamma\in\Sigma'(x) \; .
 \end{equation}
\eqref{lip} gives a constant $D>0$
such that
 \begin{equation}\label{estexp}
\exp\bigl(\bigl|F_\alpha(x)t +
\gamma\nabla\Phi_\alpha(x) (1-t)\bigr|\bigr)  \leq e^{D |\gamma|}\; .
\end{equation}
By second order Taylor-expansion, using \eqref{deltaphibound}
\begin{equation}\label{dopabab}
\bigl| (F_\alpha(x) + \gamma\nabla\Phi_\alpha(x) \bigr| \leq
\tilde{C}_3 \ep |\gamma|^2 \; .
\end{equation}
for all $\ep\in(0,\ep_0]$ and some $\tilde{C}_3>0$ independent of the choice of $B$.
By \eqref{abfallagamma1}, inserting \eqref{estexp} and \eqref{dopabab} into \eqref{mittel} and this in \eqref{Vphiminust}, using
Hypothesis \ref{Hypo2}(a), we get
\[
\left| D_2 (x)\right| = \left|V_0^{\Phi_\alpha} (x) - t_0^\Sigma(x,-i\nabla\Phi_\alpha)\right|\leq  \ep C'\;.
\]
This and \eqref{D1} give \eqref{Vphitsigma}.\\

{\sl Step 4:} We prove \eqref{weigequ2} and \eqref{eigenu} by use of Lemma \ref{HepDchi}.\\

Choosing $B \geq C_0(1+R_0+C_3)$ (depending only on $R_0$, but independent of $u$ and $E$), we have
\begin{equation}\label{CC0ep}
\left(\frac{B}{C_0}-C_3\right)\ep - E \geq \ep \, ,\qquad E\in[0,\ep R_0]\, .
\end{equation}
Let
\begin{equation}\label{OminusOplus} \O_-:= \{ x\in\Sigma\,|\, V_\ep (x) + V^{\Phi_\alpha} (x) - E < 0\}\qquad
\text{and}\qquad
\O_+ := \Sigma \setminus \O_- \; ,
\end{equation}
then from \eqref{CC0ep} combined with \eqref{VepundVPhi} it follows that $\O_-\subset \{d(x)< \ep B\}$ and
by \eqref{VepundVPhi}
\begin{equation}\label{normVepVphi}
|V_\ep (x) + V^{\Phi_\alpha}(x) | \leq \ep \,\max \{C_3,R_0\} \,\qquad\text{for all}\quad x\in \O_-\; .
\end{equation}
We define the functions $F_{\pm} : \Sigma \ra [0,\infty)$ by
\begin{equation}\label{F+def}
F_+(x) := \sqrt{\ep\id_{\{d(x)<B\ep \}}(x) + (V_\ep(x) + V^{\Phi_\alpha}(x) -E)\id_{\O_+}(x)}
\end{equation}
and
\begin{equation}\label{F-def}
F_-(x) := \sqrt{\ep\id_{\{d(x)<B\ep \}}(x) + (E- V_\ep(x) - V^{\Phi_\alpha}(x))\id_{\O_-}(x)}\; .
\end{equation}
Then $F_\pm$ are well defined and furthermore there exists a constant $C, \tilde{C}>0$ depending only of $R_0$ and $B$ such that
\begin{equation}\label{Feigenschaft}
F:=F_+ + F_- \geq C \,\sqrt{\ep} >0\, ,\qquad \left|F_-\right| \leq \tilde{C}\sqrt{\ep} \qquad \text{and}\qquad F_+^2 - F_-^2 = V_\ep +
V^{\Phi_\alpha} - E\; .
\end{equation}
The first inequality uses \eqref{VepundVPhi} combined with \eqref{CC0ep}.
By \eqref{ePhied} and \eqref{Feigenschaft}
\begin{equation}\label{HDchiab1}
\left\| F e^{\frac{\Phi_\alpha}{\ep}}u\right\|^2_{L^2(\Sigma)}\geq C\ep
\left\| e^{\frac{\Phi_\alpha}{\ep}}u\right\|^2_{L^2(\Sigma)}
\end{equation}
and
\begin{equation}\label{HDchiab2}
\left\| \tfrac{1}{F}e^{\frac{\Phi_\alpha}{\ep}}
\left(H_\ep^\Sigma - E\right)u \right\|^2_{L^2(\Sigma)} \leq
C\ep^{-1}\left\| e^{\frac{\Phi_\alpha}{\ep}}\left(H_\ep^\Sigma - E\right)u \right\|^2_{L^2(\Sigma)}  \, .
\end{equation}
Since $\supp F_-\subset\{d(x)<B\ep\}$, by \eqref{ePhied} and \eqref{Feigenschaft} there
exists a constant $C>0$ such that
\begin{equation}\label{HDchiab4}
\left\| F_- e^{\frac{\Phi}{\ep}}u\right\|^2_{L^2(\Sigma)} \leq
C\ep \left\| u\right\|^2_{L^2(\Sigma)}\, .
\end{equation}
Inserting \eqref{HDchiab1}, \eqref{HDchiab2} and
\eqref{HDchiab4} in
\eqref{FmitF-} yields \eqref{weigequ2}
uniformly with respect to $E\in (0, \ep R_0)$ and $u$.\\

{\sl Proof of (c):}\\

\eqref{thm161} follows at once from \eqref{abinK} and \eqref{abohneK}.\\
If $u$ is an eigenfunction of $H_\ep^\Sigma$ with eigenvalue $E$, then the first summand on rhs\eqref{weigequ2} vanishes. 
Thus \eqref{weigequ2} leads to \eqref{eigenu2}, for $\alpha < \alpha_0$. Using the monotonicity of $\Phi_\alpha$, \eqref{eigenu2} holds
for any $\alpha\in (0, 1]$.\\

{\sl Proof of (a):}\\

For $\hat{g}$ defined in \eqref{tildehatg} and $d_\delta= d *  j_\delta $ defined in Lemma \ref{tnullVnull}, we set 
\[ \tilde{\Phi}_{\alpha, \delta}(x) = \Bigl( 1-\frac{\alpha}{2}\Bigr) 
\Bigl( \bigl( 1 - \hat{g}\bigr)\,d(x) + \hat{g}\, d_\delta(x)\Bigr) \; . \]
Then for all $x\in\Sigma$ and $\delta<\delta_\alpha$
\begin{equation}\label{ii1} 
(1-\alpha) \,d(x) \leq \tilde{\Phi}_{\alpha, \delta}(x) \leq d(x) \; . 
\end{equation}
In fact, if $x\in K$, \eqref{ii1} is trivial, so let $x\in\Sigma\setminus K$. Then, writing 
$\tilde{\Phi}_{\alpha, \delta} = (1-\frac{\alpha}{2})(d + \hat{g}(d_\delta - d))$, \eqref{ii1} follows directly from 
\eqref{thm2-3}.

We now claim that for any fixed $\alpha\in (0,1]$ there exists $\delta_\alpha, \ep_\alpha>0$ such that for all $\delta<\delta_\alpha$
and $\ep<\ep_\alpha$
\begin{equation}\label{ii2}
V_0(x) - \tilde{t}_0^\Sigma (x, \nabla \tilde{\Phi}_{\alpha, \delta}(x)) \geq \begin{cases}
0\, , \quad & x\in \Sigma\cap d^{-1}([0,D) \\ C\, , \quad & x\in \Sigma\cap d^{-1}([D, \infty)) \; .\end{cases}
\end{equation}
If $d(x)\geq D$, this follows similar to the proof of \eqref{Vnulltab} (the regions $D\leq d(x) < D + \eta$ and $d(x)\geq D+\eta$ correspond
to Case 5 and 6 resp.).

If $d(x)<D$, we have 
$\nabla\tilde{\Phi}_{\alpha, \delta}(x) = (1-\frac{\alpha}{2})\nabla d(x)$ and thus by the convexity of $\tilde{t}_0$
 and the eikonal inequality \eqref{eikonalun}, analog to Step 2, Case 3 in the proof of (b),
\[ V_0(x) - \tilde{t}_0^\Sigma (x, \nabla \tilde{\Phi}_{\alpha, \delta}(x)) \geq \frac{\alpha}{2} V_0(x) \geq 0\, . \]
Similar to Step 3 in the proof of (b), it follows that for some $C_1, C_2>0$
\begin{equation}\label{ii3}
V_\ep(x) + V^{\tilde{\Phi}_\alpha}(x) \geq \begin{cases} -\ep C_1\, , \quad &
x\in \Sigma\cap d^{-1}([0,D) \\ C_2\, , \quad & x\in \Sigma\cap d^{-1}([D, \infty)) \; ,\end{cases}
\end{equation}
where we set $\tilde{\Phi}_\alpha:= \tilde{\Phi}_{\alpha, \delta}$ for any $\delta<\delta_\alpha$.
If $F_+, F_-$ are defined as in \eqref{F+def} and \eqref{F-def} with $\Phi_{\alpha}$ replaced by
$\tilde{\Phi}_{\alpha}$, arguments similar to those in \eqref{Feigenschaft} and below lead to
\[
\left\|  e^{\frac{\tilde{\Phi}_\alpha}{\ep}} u \right\|_{L^2(\Sigma)}
\leq C \Bigl[ \ep^{-1}\left\| e^{\frac{\tilde{\Phi}_\alpha}{\ep}}
\left(H_\ep^{\Sigma}-E\right)u\right\|_{L^2(\Sigma)} +
 \| u \|_{L^2(\Sigma)}  \Bigr] \; ,
\]
which combined with \eqref{ii1} proves \eqref{ab1}.
\end{proof}

\begin{proof}[Proof of Theorem 1.5]
This is a consequence of (the proof of) Theorem \ref{Theorem2},(b). Since $d\in\Ce^2(\overline{\Sigma})$, we can
use $\Phi$ defined in \eqref{DefPhix} instead of $\Phi_\alpha$. The arguments in Step 2, Case 1 - 3,  show that
there are constants $C_0, C_1>0$ independent of $B$, $E$ and $\ep_0>0$ such that
for all $\ep<\ep_0$ 
\[
V_0(x) - \tilde{t}_0^\Sigma(x, \nabla\Phi(x)) \geq \begin{cases} 0\, ,&\, x\in\Sigma\cap d^{-1}([0, B\ep]) \\
\frac{B}{C_0}\ep \, ,&\, x\in\Sigma\cap d^{-1}([B\ep, \infty))
\end{cases}
\]
Since $\Phi\in\Ce^2(\overline{\Sigma})$, the same arguments as in Step 3 of the proof of Thm. \ref{Theorem2},(b), show
\[
V_\ep(x) + V^\Phi(x) \geq \begin{cases}  - C_2\,\ep & \qquad\mbox{for}
\quad x\in\Sigma\cap d^{-1}([0, B\ep]) \\
\left(\frac{B}{C_0}-C_3\right)\ep &
\qquad\mbox{for}\quad  x\in\Sigma\cap d^{-1}([B\ep,\infty)) 
\end{cases}
\]
for some $C_2,C_3>0$ independent of $B$ and $E$, where $V^{\Phi}:=V_{\ep,\Sigma}^{\Phi}$ is defined in \eqref{Vphiep}.
Defining $F_-$ and $F_+$ by \eqref{F+def} and \eqref{F-def} with $\Phi_\alpha$ replaced by $\Phi$, we get \eqref{weigequ}
by use of Lemma \ref{HepDchi} and Lemma \ref{Philem}. Note that Lemma \ref{tnullVnull} is not needed and neither is the 
continuity assumption \eqref{K0stetig}.\\
\end{proof}

\end{document}